\documentclass[a4paper,10pt,reqno]{amsart}
\usepackage{amsmath}
\usepackage{cases}
\usepackage{amsfonts}
\usepackage[colorlinks,linkcolor=blue,citecolor=blue]{hyperref}
\usepackage{latexsym, amssymb, amsmath, amsthm, bbm}
\usepackage[all]{xy}
\usepackage{pgfplots}
\usepackage{enumerate}
\usepackage{mathrsfs}

\DeclareSymbolFont{EulerExtension}{U}{euex}{m}{n}
\DeclareMathSymbol{\euintop}{\mathop} {EulerExtension}{"52}
\DeclareMathSymbol{\euointop}{\mathop} {EulerExtension}{"48}

\allowdisplaybreaks[4]

\def \id{\operatorname{id}}

\def \C{\mathcal{C}}
\def \k{\Bbbk}

\def \dim{\operatorname{dim}}

\def \Ext{\operatorname{Ext}}

\def \C{\mathcal{C}}
\def \D{\Delta}

\def \span{\operatorname{span}}
\def \A{\mathcal{A}}
\def \B{\mathcal{B}}
\def \C{\mathcal{C}}
\def \D{\mathcal{D}}
\def \E{\mathcal{E}}
\def \F{\mathcal{F}}
\def \G{\mathcal{G}}
\def \X{\mathcal{X}}
\def \Y{\mathcal{Y}}

\def \W{\mathcal{W}}

\usepackage{mathtools}

\numberwithin{equation}{section}

\newtheorem{theorem}{Theorem}[section]
\newtheorem{lemma}[theorem]{Lemma}
\newtheorem{proposition}[theorem]{Proposition}
\newtheorem{corollary}[theorem]{Corollary}
\newtheorem{definition}[theorem]{Definition}
\newtheorem{example}[theorem]{Example}
\newtheorem{remark}[theorem]{Remark}
\newtheorem{question}[theorem]{Question}

\newtheorem{notation}[theorem]{Notation}

\begin{document}
\title[Hopf algebras of tame corepresentation type]{Coradically graded Hopf algebras of tame corepresentation type}
\thanks{}
\author[J. Yu]{Jing Yu}
\author[G. Liu]{Gongxiang Liu}
\address{School of Mathematical Sciences, University of Science and Technology of China, Hefei 230026, People's Republic
of China}
\email{yujing46@ustc.edu.cn}
\address{School of Mathematics, Nanjing University, Nanjing 210093, China}
\email{gxliu@nju.edu.cn}

\maketitle

\date{}
\begin{abstract}
Let $\k$ be an algebraically closed field of characteristic $0$ and let $H$ be a finite-dimensional Hopf algebra over $\k$ with the dual Chevalley property. In this paper, we give a description of the link quiver of $H$ for different corepresentation types. Moreover, we show that $\operatorname{gr}^c(H)$ is of tame corepresentation type if and only if $\operatorname{gr}^c(H)\cong (\k\langle x,y\rangle/I)^* \times H_0$ for some special ideals $I$. Using the methods of link quivers and bosonization, we then discuss which of the above ideals occur when $(\k\langle x,y\rangle/I)^* \times H_0$ is a Hopf algebra of tame corepresentation type under certain assumptions.
\end{abstract}

\maketitle
\section{Introduction}
According to Drozd's fundamental result \cite{Dro79}, every finite-dimensional algebra belongs to exactly one of three representation types: finite, tame, or wild. Inspired by this trichotomy, classifying classes of algebras by their representation type has become a central theme; see, for example, \cite{Ari05, Ari17, Ari21, DEMN99, EN01, KOS11, Rin75, Rin78}.

The classification by representation type for finite-dimensional Hopf algebras has attracted considerable interest, particularly for pointed Hopf algebras and their duals, the elementary Hopf algebras. For modular group algebras of finite groups, a block is of finite representation type if and only if the corresponding defect groups are cyclic, and it is tame if and only if $\operatorname{char}\k=2$ and its defect groups are dihedral, semidihedral, or generalized quaternion \cite{Ben98, BD82, Erd90, Hig54}. Among small quantum groups, only $u_q(\mathfrak{sl}_2)$ is tame; all others are wild \cite{Cil97, Sut94, Xia97}. Farnsteiner and his collaborators classified all cocommutative Hopf algebras according to their representation type \cite{Far06, FS02, FS07, FV00, FV03}. The classification of elementary (pointed) Hopf algebras of finite and tame (co)representation type was given by the second author and his collaborators from 2006 to 2013 \cite{LL07,HL09, Liu06, Liu13}.

Meanwhile, Hopf algebras with the (dual) Chevalley property have been intensively studied by many authors; see, for example, \cite{ABM12, AEG01, AGM17, Mom13, Li22a, Li22b, LL22, LZ19, Shi19, ZGH21}. These algebras constitute a natural generalization of elementary (or pointed) Hopf algebras. Our goal is to classify finite-dimensional Hopf algebras with the dual Chevalley property according to their corepresentation type. Here by the dual Chevalley property we mean that its coradical is a Hopf subalgebra.

In \cite{YLL23}, the authors proved that a finite-dimensional Hopf algebra $H$ with the dual Chevalley property is of finite corepresentation type if and only if it is coNakayama, if and only if the link quiver $\mathrm{Q}(H)$ of $H$ is a disjoint union of basic cycles, if and only if the link-indecomposable component $H_{(1)}$ containing $\k1$ is a pointed Hopf algebra and the link quiver of $H_{(1)}$ is a basic cycle.

This paper aims to classify Hopf algebras with the dual Chevalley property that are of tame corepresentation type. The link quiver serves as the main tool in the study of finite-dimensional Hopf algebras with the dual Chevalley property and of finite corepresentation type. Its structure can be described using multiplicative matrices and primitive matrices (see \cite{YLL23}). Let $\mathcal{S}$ denote the set of all simple subcoalgebras of a Hopf algebra $H$ with the dual Chevalley property. The set $^{\C}{\mathcal{P}}^{\D}$ of a complete family of non-trivial $(\C, \D)$-primitive matrices can be viewed as the set of arrows from vertex $D$ to vertex $C$. Define
$
{}^{\C}\mathcal{P}=\bigcup_{\D\in\mathcal{S}}{}^{\C}{\mathcal{P}}^{\D},
{\mathcal{P}}^{\D}=\bigcup_{\C\in\mathcal{S}}{}^{\C}{\mathcal{P}}^{\D}, \mathcal{P}=\bigcup_{\C\in\mathcal{S}} {}^{\C}\mathcal{P}.
$
We may also view ${\mathcal{P}^{\D}}$ as the set of arrows with start vertex $D$ and ${^{\C}\mathcal{P}}$ as the set of arrows with end vertex $C.$ Following the procedure in \cite[Section 5]{YLL23}, we characterize the link quiver of finite-dimensional Hopf algebras with the dual Chevalley property of finite or tame corepresentation type. This appears as Theorem \ref{thm:corepandquiver} in this paper.
\begin{theorem}\label{theorem1}
Let $\k$ be an algebraically closed field of characteristic 0 and let $H$ be a finite-dimensional Hopf algebra over $\k$ with the dual Chevalley property. Denote ${}^1\mathcal{S}=\{C\in\mathcal{S}\mid \k1+C\neq \k1\wedge C\}$.
\begin{itemize}
  \item[(1)]$H$ is of finite corepresentation type if and only if $\mid{}^1\mathcal{P}\mid=1$ and ${}^1\mathcal{S}=\{\k g\}$ for some group-like element $g\in G(H)$.
  \item[(2)]If $H$ is of tame corepresentation type, then one of the following two cases occurs:
  \begin{itemize}
  \item[(i)]$\mid{}^1\mathcal{P}\mid=2$ and for any $C\in {}^1\mathcal{S}$, $\dim_{\k}(C)=1$;
  \item[(ii)]$\mid{}^1\mathcal{P}\mid=1$ and ${}^1\mathcal{S}=\{C\}$ for some $C\in\mathcal{S}$ with $\dim_{\k}(C)=4$.
  \end{itemize}
  \item[(3)]If one of the following holds, then $H$ is of wild corepresentation type.
  \begin{itemize}
  \item[(i)]$\mid{}^1\mathcal{P}\mid\geq3$;
  \item[(ii)]$\mid{}^1\mathcal{P}\mid=2$ and there exists some $C\in{}^1\mathcal{S}$ with $\dim_{\k}(C)\geq 4$;
  \item[(iii)]$\mid{}^1\mathcal{P}\mid=1$ and ${}^1\mathcal{S}=\{C\}$ for some $C\in\mathcal{S}$ with $\dim_{\k}(C)\geq 9$.
  \end{itemize}
  \end{itemize}
\end{theorem}

Denote $\mathcal{S}^1=\{C\in\mathcal{S}\mid C+\k1\neq C\wedge \k1\}$. Note that$\mid{}^1\mathcal{P}\mid=\mid\mathcal{P}{}^1\mid$ and $C\in{}^1\mathcal{S}$ if and only if $S(C)\in\mathcal{S}^1$ (see Lemma \ref{lemma:P^1=1^P}). Using Theorem \ref{theorem1}, we know that if $H$ is of tame corepresentation type, then one of the following three cases occurs:
  \begin{itemize}
  \item[(i)]$\mid\mathcal{P}{}^1\mid=1$ and $\mathcal{S}{}^1=\{C\}$ for some $C\in\mathcal{S}$ with $\dim_{\k}(C)=4$;
  \item[(ii)]$\mid\mathcal{P}{}^1\mid=2$ and $\mathcal{S}{}^1=\{\k g\}$ for some $g\in G(H)$;
  \item[(iii)]$\mid\mathcal{P}{}^1\mid=2$ and $\mathcal{S}{}^1=\{\k g, \k h\}$ for some $g, h\in G(H)$.
  \end{itemize}
Furthermore, we determine the structure of finite-dimensional coradically graded Hopf algebras of tame corepresentation type. See Theorem \ref{thm:tamestructure}, which states:
\begin{theorem}\label{theorem2}
Let $\k$ be an algebraically closed field of characteristic 0 and let $H$ be a finite-dimensional Hopf algebra over $\k$ with the dual Chevalley property. Then $\operatorname{gr}^c(H)$ is of tame corepresentation type if and only if $$\operatorname{gr}^c(H)\cong (\k\langle x,y\rangle/I)^* \times H_0$$ for some ideal $I$  of the following forms:
\begin{itemize}
  \item[(1)]$I=(x^2-y^2, yx-ax^2, xy)$ for $0\neq a\in\k$;
  \item[(2)]$I=(x^2, y^2, (xy)^m-a(yx)^m)$ for $0\neq a\in\k$ and $m\geq 1$;
  \item[(3)]$I=(x^n-y^n, xy, yx)$ for $n\geq 2$;
  \item[(4)]$I=(x^2, y^2, (xy)^mx-(yx)^my)$ for $m\geq1$.
\end{itemize}
\end{theorem}
According to \cite[Theorem 4.1.2]{Bes97}, if $R$ is a Hopf algebra in ${}^{H_0}_{H_0}\mathcal{YD}$, then its bosonization $R\times H_0$ yields a Hopf algebra.
For a tame algebra $A$, the above theorem does not guarantee the existence of a finite-dimensional semisimple Hopf algebra $H^\prime$ such that $A^*$ becomes a braided Hopf algebra in ${}^{H^\prime}_{H^\prime}\mathcal{YD}$. Consequently, for the ideals $I$ listed in the theorem above, it remains unclear whether $(\k\langle x,y\rangle/I)^* \times H^\prime$ forms a Hopf algebra. To address this question, we employ the methods of link quivers and bosonization, analyzing the three cases separately.

We consider case (i) under some assumptions. See Proposition \ref{prop:I=I2}, which states:
\begin{proposition}
Let $\operatorname{gr}^c(H)\cong (\k\langle x,y\rangle/I)^* \times H_0$ be a finite-dimensional coradically graded Hopf algebra over $\k$ of tame corepresentation type.
Suppose $\mathcal{P}{}^1=\{\X\}$, $\mathcal{S}{}^1=\{C\}$ for some $C\in\mathcal{S}$ with $\dim_{\k}(C)=4$, and the invertible matrix $K$ in Lemma \ref{lem:invertibleK} is diagonal, namely
$$
K= \left(\begin{array}{cccc}
\alpha_1\\
&\alpha_2\\
&&\alpha_3\\
&&&\alpha_4
 \end{array}\right).
$$
If, in addition, $R_H$ is generated by $u, v$, then
\begin{itemize}
  \item[(1)]$I=(x^2, y^2, (xy)^m-a(yx)^m)$ for $0\neq a\in\k$ and $m\geq1$;
  \item[(2)]$\alpha_1=\alpha_4=-1$;
  \item[(3)]$a=(-1)^{m-1}\alpha_2^m$ or $a=(-1)^{m-1}\alpha_3^m$;
  \item[(4)]$\alpha_2\alpha_3$ is an $m$th primitive root of unity.
\end{itemize}
\end{proposition}

In fact, when studying the properties of finite-dimensional coradically graded Hopf algebras over $\k$ of tame corepresentation type, it suffices to consider their link-indecomposable component containing $\k 1$. This appears as Proposition \ref{HtameiffH0tame}:
\begin{proposition}
Let $H$ be a finite-dimensional coradically graded Hopf algebra over $\k$. Then $H$ is of tame corepresentation type if and only if $H_{(1)}$ is of tame corepresentation type.
\end{proposition}
With the help of the preceding proposition, we can consider cases (ii) and (iii). See Proposition \ref{prop:g=g,gnoh}, which states:
\begin{proposition}
Let $\operatorname{gr}^c(H)\cong (\k\langle x,y\rangle/I)^* \times H_0$ be a finite-dimensional coradically graded Hopf algebra over $\k$ of tame corepresentation type.
\begin{itemize}
  \item[(1)]If $\mid\mathcal{P}{}^1\mid=2$ and $\mathcal{S}{}^1=\{\k g\}$ for some $g\in G(H)$, then $I=(x^2, y^2, xy+yx)$;
  \item[(2)]If $\mid\mathcal{P}{}^1\mid=2$ and $\mathcal{S}{}^1=\{\k g, \k h\}$ for some $g, h\in G(H)$, then $I=(x^2, y^2, (xy)^m-a(yx)^m)$ for $0\neq a\in\k$ and $m\geq1$.
\end{itemize}
\end{proposition}

The organization of this paper is as follows. In Section \ref{section1}, we recall the definitions of multiplicative and primitive matrices and give a construction of a complete family of non-trivial $(\mathcal{C}, \mathcal{D})$-primitive matrices. We discuss the properties of the link quiver of Hopf algebras with the dual Chevalley property in Section \ref{section2}. Section \ref{section3} is devoted to characterizing the link quiver of Hopf algebras with the dual Chevalley property of tame corepresentation type. In Section \ref{section4}, we determine the structure of coradically graded Hopf algebras $H$ of tame corepresentation type. We show in Section \ref{section5} that $H$ is of tame corepresentation type if and only if the link-indecomposable component $H_{(1)}$ containing $\k 1$ is of tame corepresentation type. Section \ref{section6} discusses which ideals occur when $(\Bbbk \langle x,y\rangle / I)^{*} \times H_{0}$ is a finite-dimensional coradically graded Hopf algebra of tame corepresentation type under some assumptions. Finally, some examples and applications are given in Section \ref{section7}.

\section{Preliminaries} \label{section1}
Throughout this paper $\k$ denotes an \textit{algebraically closed field of characteristic $0$} and all spaces are over $\k$. We refer to \cite{Mon93} for the basics about Hopf algebras.

\subsection{Multiplicative matrices and primitive matrices}\label{subsection1.1}
In this subsection, let $(H,\Delta,\varepsilon)$ be a coalgebra over $\k$.
Denote the coradical filtration of $H$ by $\{H_n\}_{n\geq0}$ and the set of all simple subcoalgebras of $H$ by $\mathcal{S}$.

We first recall the definition of multiplicative matrices.
\begin{definition}\emph{(}\cite[Definition 2.3]{Li22a}\emph{)}
Let $(H,\Delta,\varepsilon)$ be a coalgebra over $\k$.
\begin{itemize}
  \item[(1)] A square matrix $\G=(g_{ij})_{r\times r}$ over $H$ is said to be multiplicative, if for any $1\leq i,j \leq r$, we have $\Delta(g_{ij})=\sum\limits_{t=1}^r g_{it}\otimes g_{tj}$ and $\varepsilon(g_{ij})=\delta_{i, j}$, where $\delta_{i, j}$ denotes the Kronecker notation;
  \item[(2)] A multiplicative matrix $\C$ is said to be basic, if its entries are linearly independent.
\end{itemize}
\end{definition}
Multiplicative matrices over a coalgebra generalize the notion of group-like elements. The entries of any basic multiplicative matrix $\C$ span a simple subcoalgebra $C$ of $H$. Conversely, for any simple coalgebra $C$ over $\k$, there exists a basic multiplicative matrix $\C$ whose entries span $C$ (see \cite{LZ19}, \cite{Li22a}). Moreover, by \cite[Lemma 2.4]{Li22a}, the basic multiplicative matrix of a simple coalgebra $C$ is unique up to similarity.

Next, we recall the definition of primitive matrices, which serve as a non-pointed analogue of primitive elements.
\begin{definition}\emph{(}\cite[Definition 3.2]{LZ19} and \cite[Definition 4.4]{Li22b}\emph{)}
Let $(H,\Delta,\varepsilon)$ be a coalgebra over $\k$. Suppose $\C=(c_{ij})_{r\times r}$ and $\D=(d_{ij})_{s\times s}$ are basic multiplicative matrices over $H$.
\begin{itemize}
  \item[(1)] A matrix $\X=(x_{ij})_{r\times s}$ over $H$ is said to be $(\C, \D)$-primitive, if $$\Delta(x_{ij})=\sum\limits_{k=1}^r c_{ik}\otimes x_{kj}+\sum\limits_{t=1}^s x_{it}\otimes d_{tj}$$ holds for any $1\leq i\leq r, 1\leq j\leq s$;
  \item[(2)] A primitive matrix $\X$ is said to be non-trivial, if there exists some entry of $\X$ which does not belong to the coradical $H_0$.
\end{itemize}
\end{definition}
For any matrix $\X=\left(x_{ij}\right)_{r\times s}$ over $H$, denote the matrix $\left(\overline{x_{ij}}\right)_{r\times s}$ by $\overline{\X}$, where $\overline{x_{ij}}=x_{ij}+H_0\in H/H_0$. The subspace of $H/H_0$ spanned by the entries of $\overline{\X}$ is denoted by $\operatorname{span}(\overline{\X})$.
Let $\pi: H_1 \longrightarrow H_1/H_0$ be the quotient map. For any $\overline{h}\in H_1/H_0$, define
\begin{eqnarray}\label{comodulestructure}
\rho_L(\bar{h})=(\id\otimes\pi)\Delta(h),\;\;\rho_R(\bar{h})=(\pi\otimes\id)\Delta(h).
\end{eqnarray}
Then $(H_1/H_0, \rho_L, \rho_R)$ is an $H_0$-bicomodule. Suppose that $\X_{r\times s}=\left(x_{ij}\right)_{r\times s}$ is a non-trivial $(\C, \D)$-primitive matrix.
By
\cite[Lemma 2.4]{YLL23}, $(\span(\overline{\X}), \rho_L, \rho_R)$ with the induced bicomodule structure is a simple $C$-$D$-bicomodule satisfying $\dim_{\k}(\span(\overline{\X}))=rs.$
Conversely, if $W$ is a subspace of $H_1$ such that $\overline{W}$ is a simple $C$-$D$-sub-bicomodule of $H_1/H_0$, then there exists a non-trivial $(\C, \D)$-primitive matrix $\W$ such that $\span(\overline{\W})=\overline{W}$ (see \cite[Lemma 2.10]{YLL23}).

Recall that a family $\{e_C\}_{C\in \mathcal{S}}$ in $H^*$ is called a family of \textit{coradical orthonormal idempotents} (see \cite[Section 1]{Rad78}) if $$e_C|_D=\delta_{C,D}\varepsilon|_D,\;\;\;\;e_Ce_D=\delta_{C,D}e_C\;\;\;\;
      (\text{for any}\;C,D\in\mathcal{S}),\;\;\;\;\sum\limits_{C\in\mathcal{S}}e_C=\varepsilon.$$
The existence of such a family is established in \cite[Lemma 2]{Rad78}.
Given such a family $\{e_C\}_{C\in \mathcal{S}}$ and an $H_0$-$H_0$-bicomodule with left comodule structure $\delta_L$ and right comodule structure $\delta_R$, define for $v\in V$:
$${}^Cv=v\leftharpoonup e_C=(e_C\otimes \id)\delta_L(v) ,\;\;\;v^D=e_D\rightharpoonup v=(\id\otimes e_D)\delta_R(v),$$
$${}^Cv{}^D=e_D\rightharpoonup v\leftharpoonup e_C.
$$

Using these notations, we obtain the following decompositions.
\begin{lemma}\emph{(}\cite[Lemma 2.8]{YLL23}\emph{)}\label{lemma:sum2}
As an $H_0$-$H_0$-bicomodule, $H_1/H_0=\bigoplus\limits_{C, D\in\mathcal{S}}({}^CH_1{}^D+H_0)/H_0$. Moreover, ${^C(H_1/H_0)^D}=({}^CH_1{}^D+H_0)/H_0$ holds for any $C, D\in\mathcal{S}$.
\end{lemma}

\begin{lemma}\emph{(}\cite[Corollary 2.11]{YLL23}\emph{)}\label{lemma:complete}
There exists a family $\{\X^{(\gamma)}\}_{\gamma\in\mathit{\Gamma}}$ of non-trivial $(\C, \D)$-primitive matrices such that
\begin{eqnarray}\label{equation:direct sum}
{^C(H_1/H_0)^D}=({}^CH_1{}^D+H_0)/H_0=\bigoplus\limits_{\gamma\in\mathit{\Gamma}}\span(\overline{\X^{(\gamma)}}).
\end{eqnarray}
\end{lemma}
\begin{definition}\emph{(}\cite[Definition 2.12]{YLL23}\emph{)}
A family of non-trivial $(\C, \D)$-primitive matrices $\{\X^{(\gamma)}\}_{\gamma\in\mathit{\Gamma}}$ satisfying the property of (\ref{equation:direct sum}) in Lemma \ref{lemma:complete} is said to be complete.
\end{definition}

\subsection{Constructions of a complete family of non-trivial primitive matrices}\label{subsection1.2}
Recall that a finite-dimensional Hopf algebra has the \textit{dual Chevalley property}, if its coradical $H_0$ is a Hopf subalgebra. In this paper, we still use the term \textit{dual Chevalley property} to indicate a Hopf algebra $H$ with its coradical $H_0$ as a Hopf subalgebra, even if $H$ is infinite-dimensional.

In this subsection, let $H$ be a Hopf algebra over $\k$ with the dual Chevalley property. Two matrices $\A$ and $\A^\prime$ over $H$ are \textit{similar}, denoted $\A\sim\A^\prime$, if there exists an invertible matrix $L$ over $\k$ such that $\A^\prime=L\A L^{-1}$.
For matrices $\A=(a_{ij})_{r\times s}$ and $\B=(b_{ij})_{u\times v}$ over $H$, define the operations $\A\odot \B$ and $\A\odot^\prime \B$ as follows:
 $$\A\odot \B=
\left(\begin{array}{ccc}
      a_{11}\B& \cdots &  a_{1s}\B  \\
      \vdots  & \ddots & \vdots  \\
      a_{r1}\B&  \cdots & a_{rs}\B
    \end{array}\right),\;\;
\A\odot^\prime \B=\left(\begin{array}{cccc}
      \A b_{11} &   \cdots &  \A b_{1v} \\
      \vdots &  \ddots & \vdots  \\
      \A b_{u1} &  \cdots & \A b_{uv}
    \end{array}\right).$$

By \cite[Proposition 2.6]{Li22a}, for any $(\C, \D)$-primitive matrix $\X$, there exist invertible matrices $L_{\B, \C}, L_{\B, \D}$ over $\k$ such that
\begin{eqnarray}
&&\left(\begin{array}{cc}
    L_{\B, \C}  \\
    & L_{\B, \D}
  \end{array}\right)\left(\B\odot^\prime
\left(\begin{array}{cc}
\C&\X\\
0&\D
 \end{array}\right)
 \right)
 \left(\begin{array}{cc}
    L_{\B, \C}^{-1}  \\
    & L_{\B, \D}^{-1}
  \end{array}
 \right)\notag\\
&=&\left(\begin{array}{cccccc}
    \E_1 &  &  & \X_{11} & \cdots & \X_{1u_{(\B, \D)}}  \\
    & \ddots &  & \vdots &  & \vdots  \\
     & & \E_{u_{(\B, \C)}} & \X_{u_{(\B, \C)}1} & \cdots & \X_{u_{(\B, \C)}u_{(\B, \D)}}  \\
     &  &  & \F_1 &  &   \\
     & 0 &  &  & \ddots &   \\
     &  &  &  &  & \F_{u_{(\B, \D)}}
  \end{array}\right),
\end{eqnarray}
 where $\E_1, \cdots, \E_{u_{(\B, \C)}}, \F_1, \cdots, \F_{u_{(\B, \D)}}$ are the given basic multiplicative matrices.
Combining \cite[Remark 2.5 and Lemma 2.7]{Li22a} and \cite[Remark 3.2]{LZ19}, we deduce that each $\X_{ij}$ is a $(\E_i, \F_j)$-primitive matrix. Note that cosemisimple coalgebra $BC$ (resp., $BD$) decomposes as a direct sum of simple subcoalgebras, and $u_{(\B, \C)}$ (resp., $u_{(\B, \D)}$) is precisely the number of such simple subcoalgebras. In particular, we set $L_{1, \C}=L_{\C, 1}=I$, where $I$ is the identity matrix over $\k$.

With the notations above, we have
\begin{lemma}\label{Lemma:CXno0}\emph{(}\cite[Lemma 3.2]{YLL23}\emph{)}
For any $B, C, D\in\mathcal{S}$ with basic multiplicative matrices $\B, \C, \D$ respectively. If $\X$ is a non-trivial $(\C, \D)$-primitive matrix, then
\begin{itemize}
  \item[(1)]For each $1\leq i\leq u_{(\B, \C)}$, there is some $1\leq j\leq u_{(\B, \D)}$ such that $\X_{ij}$ is non-trivial;
  \item[(2)]For each $1\leq j\leq u_{(\B, \D)}$, there is some $1\leq i\leq u_{(\B, \C)}$ such that $\X_{ij}$ is non-trivial.
  \end{itemize}
\end{lemma}

\begin{notation}
Let $\mathcal{M}$ denote the set of representatives of basic multiplicative matrices over $H$ for the similarity relation.
\end{notation}
Let ${}^1\mathcal{S}=\{C\in\mathcal{S}\mid \k1+C\neq \k1\wedge C\}$. For any $C\in{}^1\mathcal{S}$ with basic multiplicative matrix $\C\in\mathcal{M}$, Lemma \ref{lemma:complete} allows us to fix a complete family $\{\X_{C}^{(\gamma_C)}\}_{\mathit{\gamma_C\in\Gamma}_C}$ of non-trivial $(1, \C)$-primitive matrices.

Define the set
\begin{eqnarray}\label{def:1^P}
{^1\mathcal{P}}:=\bigcup\limits_{C\in {}^1\mathcal{S}}\{\X_{C}^{(\gamma_C)}\mid \gamma_C\in \mathit{\Gamma}_C\}.
\end{eqnarray}
For any non-trivial $(1, \C)$-primitive matrix $\Y\in{^1\mathcal{P}}$ and $\B\in\mathcal{M}$, we have
\begin{eqnarray}\label{equationBY}
\left(\begin{array}{cc}
I&0\\
0&L_{\B, \C}
 \end{array}\right)
\left(\B\odot^\prime
\left(\begin{array}{cc}
1&\Y\\
0&\C
 \end{array}\right)\right)
 \left(\begin{array}{cc}
I&0\\
0&L_{\B, \C}^{-1}
 \end{array}\right)
=\left(\begin{array}{cccccc}
    \B&  &  & {\Y_{ 1}} & \cdots & {\Y_{ u_{(\B, \C)}}}  \\
     &  &  & \E_{1} &  &   \\
    0&  &   &  & \ddots &   \\
     &  &  &  &  & \E_{u_{(\B, \C)}}
  \end{array}\right),
\end{eqnarray}
where $\E_1, \E_2, \cdots, \E_{u_{(\B, \C)}}\in \mathcal{M}$. By Lemma \ref{Lemma:CXno0}, $\Y_1, \Y_2, \cdots, \Y_{u_{(\B, \C)}}$ are non-trivial.
Denote the following sets:
\begin{eqnarray}\label{^BPY}
^{\B}\mathcal{P}_{\Y}:=\{\Y_{ i}\mid 1\leq i\leq u_{(\B, \C)}\},
\end{eqnarray}
\begin{eqnarray}\label{definition:^BP}
^{\B}\mathcal{P}:=\bigcup\limits_{\Y\in{{}^1\mathcal{P}}}{}^{\B}\mathcal{P}_{\Y},\;\;\; \mathcal{P}_{\Y}:=\bigcup\limits_{\B\in \mathcal{M}}{}^{\B}\mathcal{P}_{\Y}.
\end{eqnarray}
Note that $\bigcup\limits_{\Y\in{{}^1\mathcal{P}}}{}^{1}\mathcal{P}_{\Y}$ coincides with ${}^1\mathcal{P}$ defined in (\ref{def:1^P}).
Finally, define
\begin{eqnarray}\label{definition:P}
\mathcal{P}:=\bigcup\limits_{\B\in \mathcal{M}}{^{\B}\mathcal{P}}=\bigcup\limits_{\Y\in{{}^1\mathcal{P}}}\mathcal{P}_{\Y}.
\end{eqnarray}
\begin{lemma}\emph{(}\cite[Corollary 3.9]{YLL23}\emph{)}\label{coro:P_X}
With the notations in (\ref{definition:P}), then the union $\mathcal{P}=\bigcup\limits_{\Y\in{}^1\mathcal{P}}\mathcal{P}_{\Y}$ is disjoint.
\end{lemma}
As a consequence, we have the following lemma.
\begin{lemma}\emph{(}\cite[Theorem 3.10]{YLL23}\emph{)}\label{coro:BXcomplete}
Let $C, D\in \mathcal{S}$ with basic multiplicative matrices $\C, \D\in\mathcal{M}$ respectively. Denote
$${}^{\C}\mathcal{P}^{\D}:=\{\X\in \mathcal{P}\mid  \X \text{ is a non-trivial }(\C, \D)\text{-primitive matrix}\}.$$
Then it is a complete family of non-trivial $(\C, \D)$-primitive matrices. Moreover, we have $H_1/H_0=\bigoplus_{\X\in\mathcal{P}}\span(\overline{\X})$.
\end{lemma}

\section{Properties for link quiver}\label{section2}
Let $H$ be a coalgebra over $\k$.
Denote the set of all simple subcoalgebras of $H$ by $\mathcal{S}$. We first recall the concept of the link quiver.
\begin{definition}\emph{(}\cite[Definition 4.1]{CHZ06}\emph{)}
Let $H$ be a coalgebra over $\k$. The link quiver $\mathrm{Q}(H)$ of $H$ is defined as follows: the vertices of $\mathrm{Q}(H)$ are the elements of $\mathcal{S}$; for any simple subcoalgebra $C, D\in \mathcal{S}$ with $\dim_{\k}(C)=r^2, \dim_{\k}(D)=s^2$, there are exactly $\frac{1}{rs}\dim_{\k}((C\wedge D)/(C+D))$ arrows from $D$ to $C$.
\end{definition}
The following lemma reduces the problem of determining the number of arrows from vertex $D$ to vertex $C$ in the link quiver of $H$ to calculating the cardinal number of a complete family of non-trivial $(\C, \D)$-primitive matrices.
\begin{lemma}\emph{(}\cite[Corollary 2.18]{YLL23}\emph{)}\label{coro:complete number}
Let $C, D\in\mathcal{S}$ with basic multiplicative matrices $\C_{r\times r}$ and $\D_{s\times s}$, respectively.
If $\{\X^{(\gamma)}\}_{\gamma\in\mathit{\Gamma}}$ is a complete family of non-trivial $(\C, \D)$-primitive matrices, then the cardinal number
\begin{eqnarray}\label{arrownumber=primitive}
\mid \mathit{\Gamma}\mid=\frac{1}{rs}\dim_{\k}\left((C\wedge D)/(C+D)\right).
\end{eqnarray}
\end{lemma}
With the notations in subsection (\ref{subsection1.2}), denote
\begin{eqnarray*}
\mathcal{P}^{\D}:=\bigcup_{\C\in\mathcal{M}} {}^{\C}\mathcal{P}^{\D}.
\end{eqnarray*}
Now we can view $^{\C}{\mathcal{P}}^{\D}$ as the set of arrows from vertex $D$ to vertex $C$, view ${\mathcal{P}^{\D}}$ as the set of arrows with start vertex $D$ and view ${^{\C}\mathcal{P}}$ as the set of arrows with end vertex $C$.

\begin{lemma}\label{lemma:P^1=1^P}\emph{(}\cite[Lemma 4.6]{YLL23}\emph{)}
Let $H$ be a finite-dimensional non-cosemisimple Hopf algebra over $\k$ with the dual Chevalley property. Denote ${}^1\mathcal{S}=\{C\in\mathcal{S}\mid \k1+C\neq\k1\wedge C \}$, $\mathcal{S}^1=\{C\in\mathcal{S}\mid C+\k1\neq C\wedge \k1\}$. Then
\begin{itemize}
  \item[(1)]$\mid{^1\mathcal{P}}\mid\geq1$;
  \item[(2)]$\mid{^1\mathcal{P}}\mid=\mid\mathcal{P}^1\mid$;
  \item[(3)]$C\in{}^1\mathcal{S}$ if and only if $S(C)\in\mathcal{S}^1$.
  \end{itemize}
\end{lemma}

Let $\mathbb{Z}_+$ be the set of nonnegative integers. Some relevant concepts are recalled as follows.
\begin{definition}\emph{(}\cite[Definitions 2.1 and 2.2]{Ost03}\emph{)}
Let $A$ be an associative ring with unit which is free as a $\mathbb{Z}$-module.
\begin{itemize}
  \item[(1)]A $\Bbb{Z}_+$-basis of $A$ is a basis $B=\{b_{i}\}_{i\in I}$ such that $b_ib_j=\sum_{t\in I}c_{ij}^tb_t$, where $c_{ij}^t\in\Bbb{Z}_+$.
  \item[(2)]A ring with a fixed $\Bbb{Z}_+$-basis $\{b_i\}_{i\in I}$ is called a unital based ring if the following conditions hold:
  \begin{itemize}
  \item[(i)]$1$ is a basis element.
  \item[(ii)]Let $\tau: A\rightarrow \Bbb{Z}$ denote the group homomorphism defined by
  $$\tau(b_i)=\left\{
\begin{aligned}
1,~~~  \text{if} ~~~ b_i=1, \\
0,~~~  \text{if} ~~~ b_i\neq1.
\end{aligned}
\right.$$
There exists an involution $i \mapsto i^*$ of $I$ such that the induced map
$$a=\sum\limits_{i\in I}a_ib_i \mapsto a^*=\sum\limits_{i\in I}a_ib_{i^*},\;\; a_i\in \Bbb{Z}$$ is an anti-involution of $A$, and
$$\tau(b_ib_j)=\left\{
\begin{aligned}
1,~~~  \text{if} ~~~ i=j^*, \\
0,~~~  \text{if} ~~~ i\neq j^*.
\end{aligned}
\right.$$
  \end{itemize}
  \item[(3)]A fusion ring is a unital based ring of finite rank.
\end{itemize}
\end{definition}
Let $\Bbb{Z}\mathcal{S}$ be the free additive abelian group generated by the elements of $\mathcal{S}$. For our purpose, we first endow $\mathbb{Z}\mathcal{S}$ with a unital based $\Bbb{Z}_+$-ring structure.
For any $B, C\in\mathcal{S}$ with basic multiplicative matrices $\B, \C \in\mathcal{M}$,
the dual Chevalley property of $H$ implies via
\cite[Proposition 2.6(2)]{Li22a} that there exists an invertible matrix $L$ over $\k$ such that
\begin{equation}\label{equationCD}
L
(\B\odot^{\prime}\C) L^{-1}=
\left(\begin{array}{cccc}
      \E_1 & 0 & \cdots & 0  \\
      0 & \E_2 & \cdots & 0  \\
      \vdots & \vdots & \ddots & \vdots  \\
      0 & 0 & \cdots & \E_t
    \end{array}\right),
    \end{equation}
where each $\E_1, \E_2, \cdots, \E_t$ is a basic multiplicative matrix over $H$.
This induces a multiplication on $\Bbb{Z}\mathcal{S}$ defined for $B, C\in \mathcal{S}$ by
$$B\cdot C=\sum\limits_{i=1}^t E_i,$$
where $E_1, \cdots, E_t\in\mathcal{S}$ are the elements uniquely determined by the basic multiplicative matrices $\E_i\in\mathcal{M}$ in (\ref{equationCD}).
Let $S$ be the antipode of $H$. According to \cite[Theorem 3.3]{Lar71}, we obtain an anti-involution $C\mapsto S(C)$ of $\mathcal{S}$. With the multiplication and anti-involution defined above, we obtain the following lemma.
\begin{lemma}\emph{(}\cite[Proposition 4.3]{YLL23}\emph{)}\label{Prop:basedring}
Let $H$ be a Hopf algebra over $\k$ with the dual Chevalley property and let $\mathcal{S}$ be the set of all the simple subcoalgebras of $H$. Then $\Bbb{Z}\mathcal{S}$ is a unital based ring with $\Bbb{Z}_+$-basis $\mathcal{S}$.
\end{lemma}

\begin{lemma}\label{lem:BCend,BDstart}
Suppose that $\X$ is a non-trivial $(\C, \D)$-primitive matrix. For any $B\in\mathcal{S}$,
\begin{itemize}
\item[(1)]if $B\cdot C$ contains $E$ with a nonzero coefficient, then there exists an arrow in $\mathrm{Q}(H)$ with end vertex $E$;
\item[(2)]if $B\cdot D$ contains $F$ with a nonzero coefficient, then there exists an arrow in $\mathrm{Q}(H)$ with start vertex $F$.
\end{itemize}
\end{lemma}
\begin{proof}
We only prove $(1)$; the proof of $(2)$ is similar.
By \cite[Proposition 2.6]{Li22a}, we have
\begin{eqnarray}
\B\odot^\prime
\left(\begin{array}{cc}
\C&\X\\
0&\D
 \end{array}\right)
\sim \left(\begin{array}{cccccc}
    \E_1 &  &  & \X_{11} & \cdots & \X_{1u_{(\B, \D)}}  \\
    & \ddots &  & \vdots &  & \vdots  \\
     & & \E_{u_{(\B, \C)}} & \X_{u_{(\B, \C)}1} & \cdots & \X_{u_{(\B, \C)}u_{(\B, \D)}}  \\
     &  &  & \F_1 &  &   \\
     & 0 &  &  & \ddots &   \\
     &  &  &  &  & \F_{u_{(\B, \D)}}
  \end{array}\right),
\end{eqnarray}
 where $\E_1, \cdots, \E_{u_{(\B, \C)}}, \F_1, \cdots, \F_{u_{(\B, \D)}}$ are the given basic multiplicative matrices. According to Lemma \ref{Lemma:CXno0}, for each $1\leq i\leq u_{(\B, \C)}$, there is some $1\leq j\leq u_{(\B, \D)}$ such that $\X_{ij}$ is non-trivial;
 and for each $1\leq j\leq u_{(\B, \D)}$, there is some $1\leq i\leq u_{(\B, \C)}$ such that $\X_{ij}$ is non-trivial. Without loss of generality, for any $E_i$ contained in $B\cdot C$, suppose that $\X_{i1}$ is non-trivial. Note that $$\span(\overline{\X_{i1}})\subseteq ({}^{E_i}H_1{}^{F_1}+H_0)/H_0).$$ It follows from Lemma \ref{coro:complete number} that $$\dim_{\k}((E_i\wedge F_1)/(E_i+F_1))>0.$$ This means that there exists some arrow from $F_1$ to $E_i$. Thus the proof of $(1)$ is complete.
\end{proof}

For convenience, let $\mathcal{S}=\{C_i\mid i\in I\}$ denote the set of all simple subcoalgebras of $H$. For any $C_i, C_j\in\mathcal{S}$, we write $C_i\cdot C_j=\sum\limits_{t\in I}\alpha_{i,j}^tC_t$, where $\alpha_{i,j}^t\in\Bbb{Z}_+$.
Moreover, we denote $\mathcal{M}=\{\C_j\mid i\in I\}$, such that each $\C_j\in\mathcal{M}$ is the basic multiplicative matrix of $C_j\in\mathcal{S}$.
For any $\Y\in{}^1\mathcal{P}$ and $\C_i\in\mathcal{M}$, using the notations from Subsection (\ref{subsection1.2}), define $$\mathcal{P}^{\C_i}_{\Y}:=\mathcal{P}^{\C_i}\cap \mathcal{P}_{\Y}.$$ Now we obtain the following lemmas.
  \begin{lemma}\emph{(}\cite[Corollary 4.8]{YLL23}\emph{)}\label{coro:numberinPX}
Let $H$ be a finite-dimensional non-cosemisimple Hopf algebra over $\k$ with the dual Chevalley property. Then for any non-trivial $(1, \C_j)$-primitive matrix  $\Y\in{}^1\mathcal{P}$, where $\C_j\in\mathcal{M}$, we have
\begin{itemize}
  \item[(1)]$\mid{{}^{\C_i}\mathcal{P}_{\Y}}\mid\geq1$, $\mid{\mathcal{P}^{\C_i}_{\Y}}\mid\geq1$ hold for all $\C_i\in\mathcal{M}$;
  \item[(2)]$\mid\mathcal{P}_{\Y}^1\mid=1$.
  \end{itemize}
\end{lemma}
\begin{lemma}\emph{(}\cite[Lemmas 4.7 and 4.12]{YLL23}\emph{)}\label{lem:fpequation}
Let $H$ be a finite-dimensional non-cosemisimple Hopf algebra over $\k$ with the dual Chevalley property.
For any $\Y\in{}^{1}\mathcal{P}$, where $\Y$ is a non-trivial $(1, \C_j)$-primitive matrix and $\C_j\in\mathcal{M}$, let $\beta_{ij}$ be the cardinal number of $^{\C_i}\mathcal{P}_{\Y}$.  Then $\beta_{ij}=\sum\limits_{t\in I}\alpha_{i,j}^t\geq1$ and we have the following equation
 $$\sqrt{\dim_{\k}(C_k)}\left(\sum\limits_{i\in I} \sqrt{\dim_{\k}(C_i)}\right)=\sum\limits_{i\in I} \sqrt{\dim_{\k}(C_i)}\beta_{ik}.$$
\end{lemma}

\begin{lemma}\label{lemma:numberinquiver}\emph{(}\cite[Lemma 5.4]{YLL23}\emph{)}
Let $H$ be a finite-dimensional non-cosemisimple Hopf algebra over $\k$ with the dual Chevalley property. Let $\mid{}^1\mathcal{P}\mid=1$ and $C_k$ be the unique subcoalgebra contained in ${}^1\mathcal{S}$.
\begin{itemize}
\item[(1)]The number of arrows with end vertex $C_i$ in $\mathrm{Q}(H)$ is equal to $\sum\limits_{t\in I}\alpha_{ik}^t$, and the number of arrows with start vertex $C_i$ in $\mathrm{Q}(H)$ is equal to $\sum\limits_{t\in I}\alpha_{ik^*}^t$;
\item[(2)]The number of arrows from $C_t$ to $C_i$ in $\mathrm{Q}(H)$ is equal to $\alpha_{ik}^t$ and we have $\alpha_{ik}^t=\alpha_{tk^*}^i.$
\end{itemize}
\end{lemma}

Let $\mathrm{Q}(H)$ be the link quiver of $H$.
For each arrow $\X:C\rightarrow D$ in $\mathrm{Q}(H)$, let $\X^{-1}:D\rightarrow C$ denote the formal reverse. Recall that a \textit{walk} from $C$ to $D$ is a nonempty sequence of arrows $\X_1, \X_2, \cdots, \X_m$ such that there exist $\lambda_i\in\{-1, 1\}$ for which $\X_1^{\lambda_1}\X_2^{\lambda_2}\cdots \X_m^{\lambda_m}$ is a path from $C$ to $D$.
For each $C_i\in\mathcal{S} $, $\lambda_i\in\{-1, 1\}$, define
$$C_i^{\lambda_i}=\left\{
\begin{aligned}
C_i\quad,~~~  &\text{if}& ~~~ \lambda_i=1; \\
S(C_i),~~~  &\text{if}&~~~ \lambda_i=-1.
\end{aligned}
\right.
$$
Recall from \cite[section 3]{CR02} that a Hopf quiver $\mathcal{Q}(G, \chi)$ is connected if and only if the union $\cup_{\chi_\mathcal{C}\neq0}\mathcal{C}$ generates $G$. The following proposition generalizes this result.

\begin{proposition}\label{Prop:connected}
Let $H$ be a finite-dimensional non-cosemisimple Hopf algebra over $\k$ with the dual Chevalley property.
The link quiver $\mathcal{Q}(H)$ of $H$ is connected if and only if for any $D\in \mathcal{S}$, there exist $C_1, \cdots, C_n\in {}^1\mathcal{S}$ such that $C_1^{\lambda_1}\cdot C_2^{\lambda_2}\cdots C_n^{\lambda_n}$ contains $D$ with a nonzero coefficient, where $\{\lambda_i\mid1\leq i\leq n\}\subseteq\{-1, 1\}$.
\end{proposition}

\begin{proof}
Assume that for some $D \in \mathcal{S}$ there exist $C_1, \dots, C_s \in {}^{1}\mathcal{S}$ such that $C_1^{\lambda_1} \cdots C_s^{\lambda_s}$ contains $D$ with a nonzero coefficient. We construct a walk from $\k 1$ to $D$.
If $\lambda_1=1$ or $\lambda_1=-1$, we can find a walk from $\k1$ to $C_1^{\lambda_1}$.
When $\lambda_2=1$, there exists a non-trivial $(1, \C_2)$-primitive matrix $\X_2$. By Lemma \ref{lem:BCend,BDstart}, for any summand $E_2$ contained in $C_1^{\lambda_1}\cdot C_2$ with a nonzero coefficient, there exists an arrow from $E_2$ to $C_1^{\lambda_1}$. When $\lambda_2=-1$, there exists a non-trivial $(K_1S(\C_2)K_1^{-1}, 1)$-primitive matrix $\Y_2$, where $K_1$ is some invertible matrix over $\k$ such that $K_1S(\C_2)K_1^{-1}\in\mathcal{M}.$ By Lemma \ref{lem:BCend,BDstart}, for any summand $E_2$ contained in $C_1^{\lambda_1}\cdot S(C_2)$ with a nonzero coefficient, there exists some arrow from $C_1^{\lambda_1}$ to $E_2$. Thus for any summand $E_2$ contained in $C_1^{\lambda_1}\cdot C_2^{\lambda_2}$ with a nonzero coefficient, we obtain a walk from $\k1$ to $E_2$.
Proceeding by induction, we finally obtain a walk from $\k 1$ to $D$. Hence $\mathrm{Q}(H)$ is connected.

Conversely, suppose that $\mathrm{Q}(H)$ is connected. Then for any $D \in \mathcal{S}$, there exists a walk from $\k 1$ to $D$ passing through the vertices $E_0, E_1, \cdots, E_{n}$, where $E_0=\k1, E_{n}=D.$ We claim that for each $E_i$, $i\geq1$, there exists a family of $\{C_j\}_{1\leq j\leq i}$ such that $C_1^{\lambda_1}C_2^{\lambda_2}\cdots C_{i}^{\lambda_{i}}$ contains $E_i$ with a nonzero coefficient, where $C_1, \cdots, C_{i}\in {}^1\mathcal{S}$, $\lambda_1, \cdots, \lambda_{i}\in\{1, -1\}$. We prove this by induction.
For $i=1$, there exists an arrow either from $E_1$ to $\k1$ or from $\k1$ to $E_1$. If there exists an arrow from $E_1$ to $\k1$, the claim is evident.
If there exists an arrow from $\k1$ to $E_1$, then by Lemma \ref{lemma:P^1=1^P}, we have $S(E_1)\in {}^1\mathcal{S}$. Let $C_1=S(E_1);$ the claim follows.
Suppose that the claim holds for $E_{i}$, which means that there exists a family of $\{C_j\}_{1\leq j\leq i}$ such that $C_1^{\lambda_1}C_2^{\lambda_2}\cdots C_{i}^{\lambda_{i}}$ contains $E_i$ with a nonzero coefficient.
Consider $E_{i+1}$. We know that there must be an arrow from $E_i$ to $E_{i+1}$ or from $E_{i+1}$ to $E_i$. If there exists an arrow from $E_{i+1}$ to $E_i$, it follows from Lemma \ref{coro:BXcomplete} that there exists some non-trivial $(\E_i, \E_{i+1})$-primitive matrix $\X_i\in\mathcal{P}$. By the definition of $\mathcal{P}$, we know that there exists a non-trivial $(1, \F)$-primitive matrix $\Y\in{}^1\mathcal{P}$ such that $\X_i\in{}^{\E_i}\mathcal{P}_{\Y}$, where $\F\in\mathcal{M}$. Let $C_{i+1}=F,$ it follows that $E_i \cdot C_{i+1}$ contains $E_{i+1}$ with a nonzero coefficient.
If there exists an arrow from $E_i$ to $E_{i+1}$, we can find some non-trivial $(\E_i, \E_{i+1})$-primitive matrix $\X_i$. It is straightforward to show that $S(\X_i)$ is a non-trivial $(S(\E_{i+1}), S(\E_i))$-primitive matrix. This means that $$({}^{S(E_{i+1})}H_1{}^{S(E_{i})}+H_0)/H_0\neq0.$$ Let $K_1S(\E_i)K_1^{-1}, K_2S(\E_{i+1})K_2^{-1}\in\mathcal{M}$ be the basic multiplicative matrices of $S(E_i), S(E_{i+1})$, respectively, where $K_1, K_2$ are invertible matrices over $\k$. From Lemma \ref{coro:BXcomplete}, there exists a non-trivial $(K_1S(\E_i)K_1^{-1}, K_2S(\E_{i+1})K_2^{-1})$-primitive matrix $\X_i^\prime\in\mathcal{P}$.
By the definition of $\mathcal{P}$, we know that there exists a non-trivial $(1, \F)$-primitive matrix $\Y\in{}^1\mathcal{P}$ such that $\X_i^\prime\in{}^{K_1S(\E_i)K_1^{-1}}\mathcal{P}_{\Y}$. This means that $S(E_i) \cdot F$ contains $S(E_{i+1})$ with a nonzero coefficient. Let $C_0=S(F),$ Lemma \ref{Prop:basedring} yields that $C_0^{-1}E_i$ contains $E_{i+1}$ with a nonzero coefficient.
The proof is completed.
\end{proof}

\section{Corepresentation type of Hopf algebras with the dual Chevalley property}\label{section3}
Recall that a finite-dimensional algebra $A$ is said to be of \textit{finite representation type} provided there are finitely many non-isomorphic indecomposable $A$-modules.
We say that $A$ is of \textit{tame representation type} or $A$ is a \textit{tame} algebra if $A$ is not of finite representation type, whereas for any dimension $d>0$, there are finitely many $A$-$\k[T]$-bimodules $M_i$ which are free of finite rank as right $\k[T]$-modules such that all but finite number of indecomposable $A$-modules of dimension $d$ are isomorphic to
$M_i\otimes_{\k[T]}\k[T]/(T-\lambda)$ for $\lambda\in\k.$ $A$ is of \textit{wild representation type} or $A$ is a \textit{wild} algebra if there exists a finitely generated $A$-$\k[T]$-bimodule $B$ which is free as a right $\k(X, Y)$-module such that the functor $B\otimes_{\k(X, Y)}-$ from the category of finitely generated $\k(X, Y)$-modules to the category of finitely generated $A$-modules, preserves indecomposability and reflects isomorphisms.
A finite-dimensional coalgebra $C$ is said to be of \textit{finite corepresentation type}, if the dual algebra $C^*$ is of finite representation type. $C$ is defined to be of \textit{tame corepresentation type},
if $C^*$ is a tame algebra. We say that $C$ is of \textit{wild corepresentation type}, if the dual algebra $C^*$ is a wild algebra. See \cite{Erd90,SA07}.

Let $A$ (resp. $C$) be an algebra (resp. coalgebra) over $\k$ and $\{M_i\}_{ i\in I}$ be the complete set of isoclasses of simple left $A$-modules (resp. right $C$-comodules). The \textit{Ext quiver} $\Gamma(A)$ (resp. $\Gamma(C)$) of $A$ (resp. $C$) is an oriented graph with vertices indexed by $I$, and there are $\dim_{\k}\Ext^1(M_i, M_j)$ arrows from $i$ to $j$ for any $i, j\in I$.
To avoid confusion, for any Hopf algebra $H$ over $\k$, we denote the algebra's version of Ext quiver of $H$ by $\Gamma(H)^{\mathrm{a}}$ and denote the coalgebra's version of Ext quiver of $H$ by $\Gamma(H)^\mathrm{c}$.

Let us recall the definition of separated quiver.
\begin{definition}
Let $A$ be a finite-dimensional algebra over $\k$ and $\Gamma(A)=(\Gamma_0, \Gamma_1)$ be its Ext quiver, where $\Gamma_0=\{1, 2, \cdots, n\}$. The separated quiver $\Gamma(A)_{s}$ of $A$ has $2n$ vertices $\{1, 2, \cdots, n, 1^\prime, 2^\prime, \cdots, n^\prime\}$ and an arrow $i \rightarrow j^\prime$ for every arrow $i \rightarrow j$ of $\Gamma(A)$.
\end{definition}

Now we can characterize the link quiver of a finite-dimensional non-cosemisimple Hopf algebra $H$ over $\k$ with the dual Chevalley property when it is of finite or tame corepresentation type.

\begin{theorem}\label{thm:corepandquiver}
Let $\k$ be an algebraically closed field of characteristic 0 and let $H$ be a finite-dimensional Hopf algebra over $\k$ with the dual Chevalley property. Denote ${}^1\mathcal{S}=\{C\in\mathcal{S}\mid \k1+C\neq \k1\wedge C\}$.
\begin{itemize}
  \item[(1)]$H$ is of finite corepresentation type if and only if $\mid{}^1\mathcal{P}\mid=1$ and ${}^1\mathcal{S}=\{\k g\}$ for some group-like element $g\in G(H)$.
  \item[(2)]If $H$ is of tame corepresentation type, then one of the following two cases occurs:
  \begin{itemize}
  \item[(i)]$\mid{}^1\mathcal{P}\mid=2$ and for any $C\in {}^1\mathcal{S}$, $\dim_{\k}(C)=1$;
  \item[(ii)]$\mid{}^1\mathcal{P}\mid=1$ and ${}^1\mathcal{S}=\{C\}$ for some $C\in\mathcal{S}$ with $\dim_{\k}(C)=4$.
  \end{itemize}
  \item[(3)]If one of the following holds, then $H$ is of wild corepresentation type.
  \begin{itemize}
  \item[(i)]$\mid{}^1\mathcal{P}\mid\geq3$;
  \item[(ii)]$\mid{}^1\mathcal{P}\mid=2$ and there exists some $C\in{}^1\mathcal{S}$ with $\dim_{\k}(C)\geq 4$;
  \item[(iii)]$\mid{}^1\mathcal{P}\mid=1$ and ${}^1\mathcal{S}=\{C\}$ for some $C\in\mathcal{S}$ with $\dim_{\k}(C)\geq 9$.
  \end{itemize}
  \end{itemize}
\end{theorem}
\begin{proof}
Part (1) follows directly from \cite[Theorem 5.6]{YLL23}. Clearly, (2) is the negation of (3), so it suffices to prove (3).

Note that the coalgebra's version of Ext quiver $\Gamma(H)^{\mathrm{c}}$ of $H$ is the same as the algebra's version of Ext quiver $\Gamma(H^*)^{\mathrm{a}}$ of $H^*$. According to \cite[Theorem 2.1 and Corollary 4.4]{CHZ06}, the link quiver $\mathrm{Q}(H)$ of $H$ coincides with the algebra's version of Ext quiver $\Gamma(H^*)^{\mathrm{a}}$ of $H^*$. Note that $H^*$ is Morita equivalent to a basic algebra $\mathcal{B}(H^*)$.
Let $J$ be the ideal generated by all the arrows in $\mathrm{Q}(H)$. By Gabriel's theorem, there exists an admissible ideal $I$ such that $$\k\mathrm{Q}(H)/I\cong \mathcal{B}(H^*),$$ where $J^t\subseteq I \subseteq J^2$ for some integer $t\geq2$. Thus there exists an algebra epimorphism $$f:\mathcal{B}(H^*)\rightarrow \k\mathrm{Q}(H)/J^2.$$ It is enough to show that $\k\mathrm{Q}(H)/J^2$ is of wild representation type. Since the Jacobson radical of $\k\mathcal{Q}(H)/J^2$ is $J/J^2$, we know that $\k\mathrm{Q}(H)/J^2$ is an artinian algebra with radical square zero.
Now assume on the contrary that $\k\mathrm{Q}(H)/J^2$ is of tame representation type. It follows from the proof of \cite[X.2 Theorem 2.6]{ARS95} that the separated quiver of $\k\mathrm{Q}(H)/J^2$ coincides with the quiver of the hereditary algebra
$$\sum=\left(\begin{array}{cc}
    (\k\mathcal{Q}(H)/J^2)/(J/J^2)   & 0 \\
    J/J^2  & (\k\mathcal{Q}(H)/J^2)/(J/J^2)
  \end{array}\right).$$
Note that $\k\mathrm{Q}(H)/J^2$ and $\sum$ are stably equivalent, it follows that $\k\mathrm{Q}(H)/J^2$ is of tame representation type if and only if $\sum$ is of tame representation type.
This means that $\mathrm{Q}(H)_s$ of $\k\mathrm{Q}(H)/J^2$ is a finite disjoint union of Euclidean diagrams.\\
\begin{itemize}
  \item[(i)]If $\mid{}^1\mathcal{P}\mid\geq3$, we deal with this situation through classified discussion.
   \begin{itemize}
  \item[(a)]Suppose that there exists some $C\in{}^1\mathcal{S}$ such that $\mid{}^1\mathcal{P}{}^{\C}\mid\geq3$. Then the separated quiver $\mathrm{Q}(H)_s$ must contain
$$
\begin{tikzpicture}
\filldraw [black] (1,0) circle (2pt) node[anchor=west]{$\k 1^\prime$};
\filldraw [black] (-1,0) circle (2pt)node[anchor=east]{$C$};
\draw[thick, ->] (-0.9,0.2).. controls (-0.1,0.2) and (0.1,0.2) .. node[anchor=south]{}(0.9,0.2) ;
\draw[thick, ->] (-0.9,0).. controls (-0.1,0) and (0.1,0) .. node[anchor=south]{}(0.9,0) ;
\draw[thick, ->] (-0.9,-0.2).. controls (-0.1,-0.2) and (0.1,-0.2) .. node[anchor=south]{}(0.9,-0.2);
\end{tikzpicture}
$$
as a sub-quiver. The underlying graph of this sub-quiver is not a Euclidean diagram. It turns out that $H$ is of wild corepresentation type.
  \item[(b)]Suppose that there exist some $C_1, C_2\in{}^1\mathcal{S}$ such that $\mid{}^1\mathcal{P}{}^{\C_1}\mid\geq2$ and $\mid{}^1\mathcal{P}{}^{\C_2}\mid\geq1$.
      Then the separated quiver $\mathrm{Q}(H)_s$ must contain
$$
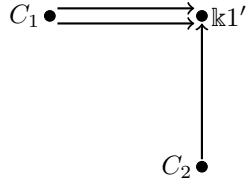
\begin{tikzpicture}
\filldraw [black] (1,0) circle (2pt) node[anchor=west]{$\k 1^\prime$};
\filldraw [black] (-1,0) circle (2pt)node[anchor=east]{$C_1$};
\filldraw [black] (1,-2) circle (2pt)node[anchor=east]{$C_2$};
\draw[thick, ->] (-0.9,0.1).. controls (-0.1,0.1) and (0.1,0.1) .. node[anchor=south]{}(0.9,0.1) ;
\draw[thick, ->] (1,-1.9).. controls (1,-1.9) and (1,-1.9) .. node[anchor=south]{}(1,-0.1) ;
\draw[thick, ->] (-0.9,-0.1).. controls (-0.1,-0.1) and (0.1,-0.1) .. node[anchor=south]{}(0.9,-0.1);
\end{tikzpicture}
$$
as a sub-quiver. The underlying graph of this sub-quiver is not a Euclidean diagram and thus $H$ is of wild corepresentation type.
  \item[(c)]Suppose that there exist some $C_1, C_2, C_3\in{}^1\mathcal{S}$ such that $\mid{}^1\mathcal{P}{}^{\C_i}\mid\geq1$ for any $1\leq i\leq 3.$ This means that for any $1\leq i\leq 3,$ there exists some non-trivial $(1, \C_i)$-primitive matrix $\X_i\in{}^1\mathcal{P}.$ Combining Lemmas \ref{coro:P_X} and \ref{coro:numberinPX}, for any $1\leq i\leq 3,$ we know that $$\mid\mathcal{P}{}^{\C_i}\mid\geq \mid\mathcal{P}_{\X_1}^{\C_i}\mid+\mid\mathcal{P}_{\X_2}^{\C_i}\mid+\mid\mathcal{P}_{\X_3}^{\C_i}\mid\geq 3.$$
      In such a case, there exist at least $3$ vertexes which are the start vertex of $3$ arrows and $1$ vertex which is the end vertex of $3$ arrows in the separated quiver $\mathrm{Q}(H)_s$.
 As a result, the underlying diagram of $\mathrm{Q}(H)_s$ is not a Euclidean diagram and $H$ is of wild corepresentation type.
  \end{itemize}
  \item[(ii)]Suppose that ${}^1\mathcal{P}=\{\X, \Y\},$ where $\X$ is a non-trivial $(1, \C)$-primitive matrix and $\Y$ is a non-trivial $(1, \D)$-primitive matrix for some $C, D\in\mathcal{S}$ with $\dim_{\k}(C)\geq4$.
      \begin{itemize}
  \item[(a)]If $\dim_{\k}(C)\geq 9$, it follows from Lemma \ref{lem:fpequation} that there exists some $E\in\mathcal{S}$ such that $\mid{}^{\E}\mathcal{P}_{\X}\mid\geq4$.
  According to Lemmas \ref{coro:P_X} and \ref{coro:numberinPX}, we know that $$\mid{}^{\E}\mathcal{P}\mid=\mid{}^{\E}\mathcal{P}_{\X}\mid+\mid{}^{\E}\mathcal{P}_{\Y}\mid\geq5.$$
  This implies that $\mathrm{Q}(H)_s$ contains at least one vertex $E$ which is the end vertex of at least 5 arrows. It follows that the underlying graph of this sub-quiver is not a union of Euclidean diagram, and consequently $H$ is of wild corepresentation type.
  \item[(b)]If $\dim_{\k}(C)=4$, Lemma \ref{lem:fpequation} implies that there exists some $E\in\mathcal{S}$ such that $\mid{}^{\E}\mathcal{P}_{\X}\mid\geq3$. If $\mid{}^{\E}\mathcal{P}_{\X}\mid\geq4$, as in the case of $\dim_{\k}(C)\geq 9$, $\mathrm{Q}(H)_s$ contains at least one vertex $E$ which is the end vertex of at least 5 arrows. This indicates $H$ is of wild corepresentation type.
      If $\mid{}^{\E}\mathcal{P}_{\X}\mid=3$, using Lemma \ref{lem:fpequation}, we have
      \begin{eqnarray} \label{eq:EC=C1+C2+C3}
      E\cdot C=C_1+C_2+C_3
      \end{eqnarray}
      for some $C_1, C_2, C_3\in\mathcal{S}$.
      According to Lemma \ref{lemma:numberinquiver}, we know that for any $1\leq i\leq3$, $C_i\cdot S(C)$ contains $E$ with a nonzero coefficient.
      Suppose that $\sqrt{\dim_{\k}(E)}=n.$ If for any $1\leq i\leq3$, we have $$C_i\cdot S(C)=E.$$ It means that $$\sqrt{\dim_{\k}(C_1)}=\sqrt{\dim_{\k}(C_2)}=\sqrt{\dim_{\k}(C_3)}=\frac{n}{2}.$$
      But (\ref{eq:EC=C1+C2+C3}) implies that $2n=\frac{3}{2}n$, which is impossible.
      Thus there exists at least one $C_j$ such that $C_j\cdot S(C)$ contains some $F\in\mathcal{S}$ with a nonzero coefficient besides $E$, where $1\leq j\leq 3$.
      Combining Lemmas \ref{coro:P_X} and \ref{coro:numberinPX}, we have $$\mid{}^{\E}\mathcal{P}\mid\geq \sum\limits_{i=1}^3\mid{}^{\E}\mathcal{P}_{\X}^{C_i}\mid+\mid{}^{\E}\mathcal{P}_{\Y}\mid\geq 4$$
      and $$\mid\mathcal{P}^{\C_j}\mid\geq\mid{}^{\E}\mathcal{P}^{\C_j}\mid+\mid{}^{\F}\mathcal{P}^{\C_j}\mid\geq2.$$
   As a result, there exist at least one vertex which is the end vertex of $4$ arrows and one vertex which is the start vertex of $4$ arrows in $\mathrm{Q}(H)_s$.
 It is easy to see that $H$ is of wild corepresentation type.
  \end{itemize}
  \item[(iii)]
  \begin{itemize}\item[(a)]Note that if $\dim_{\k}(C)\geq16$, it follows from Lemma \ref{lem:fpequation} that there exists some $E\in\mathcal{S}$ such that $\mid{}^{\E}\mathcal{P}\mid\geq5$. This means that the separated quiver $\mathrm{Q}(H)_s$ contains a vertex which is the end vertex of $5$ arrows and it cannot be a finite disjoint union of Euclidean diagram. We know that $H$ is of wild corepresentation type.
   \item[(b)]If $\dim_{\k}(C)=9$, it follows from Lemma \ref{lem:fpequation} that there exists some $E\in\mathcal{S}$ such that $\mid{}^{\E}\mathcal{P}\mid\geq4$.
   If $\mid{}^{\E}\mathcal{P}\mid\geq5$, a similar argument shows that $H$ is of wild corepresentation type. We only need to consider the case that $\mid{}^{\E}\mathcal{P}\mid=4$.
   In this case, Lemma \ref{lem:fpequation} implies that
   \begin{eqnarray}\label{EC=C1+C2+C3+C4}
   E\cdot C=C_1+C_2+C_3+C_4,
   \end{eqnarray}
   where $C_i\in\mathcal{S}$ for $1\leq i\leq 4.$
   Applying Lemma \ref{lemma:numberinquiver} yields that for any $1\leq i\leq4$, $C_i\cdot S(C)$ contains $E$ with a nonzero coefficient.
      Suppose that $\sqrt{\dim_{\k}(E)}=n.$ If for any $1\leq i\leq4$, we have $$C_i\cdot S(C)=E.$$ It means that $$\sqrt{\dim_{\k}(C_i)}=\frac{n}{3},$$
      for $1\leq i\leq 4$.
      But (\ref{EC=C1+C2+C3+C4}) implies that $3n=\frac{4}{3}n$, which leads to a contradiction.
      Thus there exists at least one $C_j$ such that $C_j\cdot S(C)$ contains some $F\in\mathcal{S}$ with a nonzero coefficient besides $E$, where $1\leq j\leq 4$.
      A similar argument shows that $\mathrm{Q}(H)_s$ contains at least one vertex which is the end vertex of $4$ arrows and one vertex which is the start vertex of $4$ arrows.
The underlying graph of this sub-quiver is not Euclidean. Consequently, $H$ is of wild corepresentation type.
\end{itemize}
  \end{itemize}
\end{proof}
According to Theorem \ref{thm:corepandquiver}, we can prove \cite[Conjecture 4.11 (1)]{YLL23} when $H$ is of finite or tame corepresentation type.
\begin{corollary}
Let $H$ be a finite-dimensional non-cosemisimple Hopf algebra over $\k$ with the dual Chevalley property of finite or tame corepresentation type. Then we have $\mid{}^{1}\mathcal{P}\mid \big| \mid{}^{\C}\mathcal{P}\mid,$ for all $\C\in\mathcal{M}.$
\end{corollary}
\begin{proof}
Note that for any $C\in{}^1\mathcal{S}$, if $\dim_{\k}(C)=1$, it follows from \cite[Proposition 4.9]{YLL23} that $$\mid{}^{1}\mathcal{P}\mid=\mid{}^{\C}\mathcal{P}\mid.$$
If $\mid{}^1\mathcal{P}\mid=1$ and $\dim_{\k}(C)=4$, where $C\in{}^1\mathcal{S}$. Lemma \ref{lem:fpequation} gives
 $$1=\mid{}^{1}\mathcal{P}\mid \big| \mid{}^{\C}\mathcal{P}\mid.$$
\end{proof}

\section{Coradically graded Hopf algebras of tame corepresentation type}\label{section4}
The main aim of this section is to describe the structure of coradically graded Hopf algebras of tame corepresentation type.

Let $H, H^\prime$ be Hopf algebras and $\pi:H\rightarrow H^\prime$ and $i: H^\prime \rightarrow H$ Hopf homomorphisms. Assume that $\pi\circ i=id_{H^\prime}$, so that $\pi$ is surjective and $i$ is injective.
Define $$R:=\{h\in H\mid (id\otimes\pi)\Delta (h)=h\otimes 1\}.$$
According to \cite[Theorem 3]{Rad85}, we know that
$$
H\cong R\times H^\prime
$$
as Hopf algebras, where ``$\times$" was called \textit{biproduct} in \cite{Rad85} and bosonization in \cite{Mar94}. As a linear space, $ R\times H^\prime=R\otimes H^\prime.$ Its multiplication and comultiplication are the usual smash product and smash coproduct respectively. In addition, $R$ is a braided Hopf algebra in ${}^{H^\prime}_{H^\prime}\mathcal{YD}$, the category of Yetter-Drinfeld modules over $H^\prime$. See, for example, \cite{AS98, Mar94, Rad85}.

Recall that a finite-dimensional Hopf algebra $H$ over $\k$ is said to have the \textit{Chevalley property}, if radical $Rad(H)$ is a Hopf ideal. According to \cite[Propersition 4.2]{AEG01}, $H$ has the Chevalley property if and only if $H^*$ has the dual Chevalley property.

Let $H$ be a finite-dimensional Hopf algebra with the Chevalley property and $J_H$ its Jacobson radical. Denote $\operatorname{gr}^a(H)$ its radically graded algebra, i.e., $\operatorname{gr}^a(H)=H/J_H\oplus J_H/J_H^2\oplus\cdots \oplus J_H^{m-1},$ if $J_H^{m}=0$. According to \cite[Lemma 5.1]{Liu06}, we know that $\operatorname{gr}^a(H)$ is a radically graded Hopf algebra. Clearly, $H/J_H=\operatorname{gr}^a(H)(0)$ is a Hopf subalgebra of $\operatorname{gr}^a(H)$ and there exists a natural Hopf algebra epimorphism $$\pi^a:\operatorname{gr}^a(H)\rightarrow H/J_H$$ with a retraction of the inclusion. Define $$A_{H}:=\{h\in \operatorname{gr}^a(H)\mid (id\otimes\pi^a)\Delta (h)=h\otimes 1\}.$$ By \cite[Theorem 3]{Rad85}, we know that
$$
\operatorname{gr}^a(H)\cong A_H\times H/J_H.
$$
as Hopf algebras.
\begin{proposition}\label{prop:samerep}
Let $\k$ be an algebraically closed field of characteristic 0 and let $H$ be a finite-dimensional Hopf algebra over $\k$ with the Chevalley property.
Then
\begin{itemize}
  \item[(1)]$A_H$ and $\operatorname{gr}^a(H)$ have the same representation type;
  \item[(2)]$A_H$ is a local graded Frobenius algebra.
\end{itemize}
\end{proposition}
\begin{proof}
\begin{itemize}
  \item[(1)]Note that as an algebra, $$\operatorname{gr}^a(H)\cong A_H\# H/J_H,$$ and the multiplication of $A_H \#H/J_H$ is usual smash product. Since $H/J_H$ is a finite-dimensional semisimple Hopf algebra, it follows from \cite[Theorem 3.3]{LR88} that $H/J_H$ is cosemisimple. Thus $(1)$ is a direct consequence of \cite[Theorem 4.5]{Liu06}.
  \item[(2)]This can be obtained by the same reason in the proof of \cite[Proposition 5.3 (ii)]{Liu06}.
\end{itemize}
\end{proof}

Let $H$ be a finite-dimensional Hopf algebra with the dual Chevalley property. Denote by $\operatorname{gr}^c(H)$ its coradically graded Hopf algebra of $H$, i.e., $\operatorname{gr}^c(H)=\bigoplus_{n\geq0}H_n/H_{n-1}$, where $H_{-1}=0.$ In fact, there exists a natural Hopf algebra epimorphism $$\pi^c:\operatorname{gr}^c(H)\rightarrow H_0$$ with a retraction of the inclusion. Define $$R_{H}:=\{h\in \operatorname{gr}^c(H)\mid (id\otimes\pi^c)\Delta (h)=h\otimes 1\}.$$
It follows from \cite[Theorem 3]{Rad85} that
$$
\operatorname{gr}^c(H)\cong R_H\times H_0
$$
as Hopf algebras.
The following theorem gives the structure of coradically graded Hopf algebras of tame corepresentation type.
\begin{theorem}\label{thm:tamestructure}
Let $\k$ be an algebraically closed field of characteristic 0 and let $H$ be a finite-dimensional Hopf algebra over $\k$ with the dual Chevalley property. Then $\operatorname{gr}^c(H)$ is of tame corepresentation type if and only if $$\operatorname{gr}^c(H)\cong (\k\langle x,y\rangle/I)^* \times H_0$$ for some ideal $I$  of the following forms:
\begin{itemize}
  \item[(1)]$I=(x^2-y^2, yx-ax^2, xy)$ for $0\neq a\in\k$;
  \item[(2)]$I=(x^2, y^2, (xy)^m-a(yx)^m)$ for $0\neq a\in\k$ and $m\geq 1$;
  \item[(3)]$I=(x^n-y^n, xy, yx)$ for $n\geq 2$;
  \item[(4)]$I=(x^2, y^2, (xy)^mx-(yx)^my)$ for $m\geq1$.
\end{itemize}
\end{theorem}
\begin{proof}
``If part": Combining \cite[Theorem 3.1]{Liu06} and \cite[Lemma 4.2]{Liu13}, we know that $\k\langle x,y\rangle/I$ is a tame algebra. Since a finite-dimensional Hopf algebra $H_0$ is semisimple if and only if it is cosemisimple, the desired conclusion follows from \cite[Theorem 4.5]{Liu06}.

``Only if part":Using Proposition \ref{prop:samerep}, we know that $\operatorname{gr}^a(H^*)$ is of tame representation type if and only if $A_{H^{*}}$ is of tame representation type.
Since $$\operatorname{gr}^c(H)\cong (\operatorname{gr}^a(H^*))^*$$ as Hopf algebra, one can conclude that $\operatorname{gr}^c(H)$ is of tame corepresentation type if and only if $A_{H^{*}}$ is of tame representation type. According to \cite[Theorem 3.1]{Liu06} and \cite[Lemma 4.2]{Liu13}, as a tame local graded Frobenius algebra, $$A_{H^{*}}\cong\k\langle x,y\rangle/I.$$
It follows from \cite[Theorem 5.1]{Mol77} that
$$\operatorname{gr}^c(H)\cong (\operatorname{gr}^a(H^*))^*\cong (A_{H^*}\times H^*/J_{H^*})^*\cong (A_{H*})^*\times H_0.$$
\end{proof}
According to \cite[Theorem 4.1.2]{Bes97}, if $R$ is a Hopf algebra in ${}^{H_0}_{H_0}\mathcal{YD}$, then the bosonization $R \times H_0$ forms a Hopf algebra. However, for a tame local graded Frobenius algebra $A$, the above theorem does not guarantee that there exists a finite-dimensional semisimple Hopf algebra $H^\prime$ such that $A^{*}$ is a braided Hopf algebra in ${}^{H^\prime}_{H^\prime}\mathcal{YD}$. In other words, for the ideals $I$ listed in the above theorem, it remains unclear whether $(\Bbbk \langle x,y\rangle / I)^{*} \times H^\prime$ is a Hopf algebra.

\begin{question}
For a tame local graded Frobenius algebra $A$, give an efficient method to determine that whether there is a cosemisimple Hopf algebra $H^\prime$ satisfying $A$ is a braided Hopf algebra in ${}^{H^\prime}_{H^\prime}\mathcal{YD}$. If such $H^\prime$ exists, then find all of them.
\end{question}
The question above exactly recovers \cite[Problem 5.1]{Liu06}.
We will discuss this question in the subsequent sections under certain conditions.

\section{Link-indecomposable component containing $\k 1$}\label{section5}
Let us first introduce the notion of link-indecomposable components.
\begin{definition}\emph{(}\cite[Definition 1.1]{Mon95}\emph{)}
A subcoalgebra $H^\prime$ of coalgebra $H$ is called \textit{link-indecomposable} if the link quiver $\mathcal{Q}(H^\prime)$ of $H^\prime$ is connected (as an undirected graph).
A \textit{link-indecomposable component} of $H$ is a maximal link-indecomposable subcoalgebra.
\end{definition}

We now characterize the coradical of $H_{(1)}$.
\begin{lemma}\label{lem:generate}
Let $H$ be a finite-dimensional Hopf algebra over $\k$ with the dual Chevalley property. Then the coradical of the link-indecomposable component $H_{(1)}$ containing $\k 1$ is generated by $\{\span(C)\mid C\in {}^1 \mathcal{S}\}\cup\{\span(S(C))\mid C\in {}^1 \mathcal{S}\}.$
\end{lemma}
\begin{proof}
It is directly from \cite[Theorem 4.8 (3)]{Li22a} that $H_{(1)}$ is a link-indecomposable Hopf algebra. This means that the link quiver $\mathcal{Q}(H_{(1)})$ of $H_{(1)}$ is connected. Using Proposition \ref{Prop:connected}, we can complete the proof.
\end{proof}
Now we discuss the relation between the corepresentation type of $H$ and that of $H_{(1)}$.
\begin{lemma}\label{lem:1tame}
Let $H$ be a finite-dimensional Hopf algebra over $\k$ with the dual Chevalley property of tame corepresentation type.
Then the link-indecomposable component $H_{(1)}$ containing $\k1$ is of tame corepresentation type.
\end{lemma}
\begin{proof}
Since $H$ is of tame corepresentation type, it follows from Theorem \ref{thm:corepandquiver} that either $\mid {}^1\mathcal{P}\mid>1$ or $\dim_{\k}(C)>1$ for $C\in{}^1\mathcal{S}$.
Hence $H_{(1)}$ is not of finite corepresentation type.
On the other hand, there is an inclusion from the category of finite-dimensional right $H_{(1)}$-comodules to that of finite-dimensional right $H$-comodules. Suppose that $H_{(1)}$ is of wild corepresentation type. It follows that $H_{(1)}^*$ is a wild algebra. Hence by \cite[Theorem 1.11]{SA07}, $H^*$ is a wild algebra, which means that $H$ is of wild corepresentation type. This leads to a contradiction. Therefore, $H_{(1)}$ is of tame corepresentation type by the fundamental result of \cite{Dro79}.
\end{proof}
In the following part, let $H=\bigoplus_{i=0}^nH(i)$ be a finite-dimensional coradically graded Hopf algebra over $\k$. Then it has the dual Chevalley property. Note that there exists a natural Hopf algebra epimorphism $$\pi:H\rightarrow H_0$$ with a retraction of the inclusion.
We now give a more precise description of the structure of $R_H$, where $$R_{H}=\{h\in H\mid (id\otimes\pi)\Delta (h)=h\otimes 1\}.$$

\begin{lemma}\label{lem:Rin1}
Let $H$ be a finite-dimensional coradically graded Hopf algebra over $\k$. Then we have $R_H\subseteq H_{(1)}$.
\end{lemma}
\begin{proof}
Define an equivalence relation on $\mathcal{S}$ by declaring $C$ and $D$ are equivalent if $CH_{(1)}=DH_{(1)}$. Let $\mathcal{S}_{0}\subseteq \mathcal{S}$ be a full set of chosen non-related representatives with respect to this equivalence relation.
By \cite[Corollary 4.10]{Li22a}, we have $$H=\bigoplus_{C\in S_0} C H_{(1)}.$$
Take any nonzero $x\in C H_{(1)}$ with $C\in\mathcal{S}_0\setminus\{\k1\}$.
By \cite[Theorem 4.8 (3)]{Li22a},
\begin{eqnarray*}
(id\otimes \pi)\Delta(x)=(id\otimes \pi)\Delta(\sum\limits_{i=1}^n c_iy_i)
\subseteq(id\otimes \pi)(C H_{(1)}\otimes CH_{(1)})
\end{eqnarray*}
Since $H$ is a coradically graded, it follows that $\pi(CH_{(1)}(i))=0$ for $i\geq 1,$ where $H_{(1)}(i)=H_{(1)}\cap H(i).$
By Lemma \ref{Prop:basedring}, $\Bbb{Z}\mathcal{S}$ is a unital based ring, and $1\notin CH_{(1)}.$ Hence $x\notin R_H$, so $R_H\subseteq H_{(1)}$.
\end{proof}
Moreover, $H_{(1)}=\bigoplus_{i=0}^nH_{(1)}(i)$ is also a finite-dimensional coradically graded Hopf algebra over $\k$, where $H_{(1)}(i)=H_{(1)}\cap H(i).$
Let $$\pi^\prime: H_{(1)}\rightarrow (H_{(1)})_0$$ be a natural Hopf algebra epimorphism with a retraction of the inclusion and  $$R^\prime=\{r\in H_{(1)}\mid (id\otimes \pi^\prime)\Delta(r)=r\otimes 1\}.$$
\begin{lemma}\label{lem:R=Rprime}
 With the notations above, we have $R^\prime=R_H$ and $H_{(1)}\cong R_H \times (H_{(1)})_0$.
\end{lemma}
\begin{proof}
Since $\pi^\prime=\pi\mid_{H_{(1)}},$ we have $R^\prime\subseteq R_H.$ By Lemma \ref{lem:Rin1}, $R^\prime=R_H.$
The conclusion then follows directly from \cite[Theorem 3]{Rad85}.
\end{proof}

With the help of the preceding lemmas, we can now prove the following.
\begin{proposition}\label{HtameiffH0tame}
Let $H$ be a finite-dimensional coradically graded Hopf algebra over $\k$. Then $H$ is of tame corepresentation type if and only if $H_{(1)}$ is of tame corepresentation type.
\end{proposition}
\begin{proof}
The ``only if" implication follows immediately from Lemma \ref{lem:1tame}. For the ``if" implication, since $H_{(1)}$ is of tame corepresentation type, it follows from Theorem \ref{thm:tamestructure} that $$H_{(1)}\cong (\k\langle x,y\rangle/I)^*\times (H_{(1)})_0$$ for some $I$ listed in Theorem \ref{thm:tamestructure}. By Lemma \ref{lem:R=Rprime}, we have $$H\cong (\k\langle x,y\rangle/I)^*\times H_0,$$  and Theorem \ref{thm:tamestructure} mplies that $H$ is of tame corepresentation type.
\end{proof}
The above proposition shows that when studying finite-dimensional coradically graded Hopf algebras over $\k$ of tame corepresentation type, it suffices to consider the link-indecomposable component containing $\k 1$.

\section{Characterization of $R_H$}\label{section6}
In this section, we discuss which ideal in Theorem 5.2 occurs when $(\k \langle x,y\rangle / I)^{*} \times H_0$ is a finite-dimensional coradically graded Hopf algebra of tame corepresentation type.

Let $H$ be a finite-dimensional Hopf algebra with the dual Chevalley property. Denote the coradical filtration of $H$ by $\{H_n\}_{n\geq0}$ and the set of all simple subcoalgebras of $H$ by $\mathcal{S}$. There exists a natural Hopf algebra epimorphism $$\pi:\operatorname{gr}^c(H)\rightarrow H_0$$ with a retraction of the inclusion $$i:H_0\rightarrow \operatorname{gr}^c(H).$$
Denote $$R_{H}:=\{h\in \operatorname{gr}^c(H)\mid (id\otimes\pi)\Delta (h)=h\otimes 1\}.$$
We will give a more precise description of the structure of $R_H$.

First, recall some properties of the biproduct. Set $\Pi=id\ast (i\circ S\circ \pi)$, where $S$ is the antipode of $\operatorname{gr}^c(H)$ and $\ast$ denotes the convolution product. According to \cite[Theorem 3]{Rad85}, we know that $R_H=\Pi(\operatorname{gr}^c(H))$, and $R_H$ has a unique coalgebra structure such that $\Pi$ is a coalgebra map. Let $j: R_H\rightarrow \operatorname{gr}^c(H)$ be the inclusion. Then the map $$\eta: R_H\times H_0 \rightarrow \operatorname{gr}^c(H),\;\;\; r\times h\mapsto r j(h)$$ is an isomorphism of Hopf algebras.

Moreover, it follows from \cite[Theorem 2 (b)]{Rad85} that the following diagrams
\begin{displaymath}
\xymatrix{
&\ar[dl]^\Pi \operatorname{gr}^c(H)\ar[dr]^\pi  \\
R_H&& H_0\\
&\ar[ul]^{\Pi_{R_H}}R_H\times H_0\ar[uu]_{\eta}\ar[ur]^{\pi_{H_0}}
}\;\;\;\;\;\;
\xymatrix{
& \operatorname{gr}^c(H)  \\
R_H\ar[ur]^j\ar[dr]_{j_R}&& \ar[ul]_i\ar[dl]^{i_{H_0}} H_0\\
&R_H\times H_0\ar[uu]_{\eta}
}
\end{displaymath}
commute, where
\begin{eqnarray*}
\Pi_{R_H}&:&r\times h\mapsto r\varepsilon(h),\\
j_{R_H}&:&r\mapsto r\times 1,\\
i_{H_0}&:&h\mapsto1\times h,\\
\pi_{H_0}&:& r\times h\mapsto \varepsilon(r)h,
\end{eqnarray*}
for all $h\in H_0, r\in R_H.$
\begin{lemma}\label{lem:deltaRH}
For any $r\in R_H$,  $\Delta_{R_H}(r)=((\Pi_R\circ\eta^{-1}) \otimes id)\Delta(r),$ where $\Delta$ and $\Delta_{R_H}$ are the comultiplications of $H$ and $R_H$, respectively.
\end{lemma}
\begin{proof}
From the proof of \cite[Theorem 3]{Rad85}, we have $$\Delta_{R_H}(r)=(\Pi\otimes id)\Delta(r).$$
Then
\begin{eqnarray*}
\Delta_{R_H}(r)&=&(\Pi\otimes id)\Delta(r)\\
&=&(\Pi\otimes id)\Delta(\eta(r\times 1))\\
&=&(\Pi\otimes id)(\eta\otimes\eta)\Delta^{\prime}(r\times 1)\\
&=&(\Pi_{R_H} \otimes \eta)\Delta^\prime(r\times 1)\\
&=&(\Pi_{R_H} \otimes \eta)(\eta^{-1}\otimes \eta^{-1})\Delta(r)\\
&=&((\Pi_{R_H}\circ\eta^{-1}) \otimes id)\Delta(r),
\end{eqnarray*}
where $\Delta^{\prime}$ is the comultiplication of ${R_H}\times H_0$.
\end{proof}

As stated in the previous section, $\operatorname{gr}^a(H^*)$ is a finite-dimensional radically graded Hopf algebra over $\k$.
There exists a natural Hopf algebra epimorphism $$\tau:\operatorname{gr}^a(H^*)\rightarrow H^*/J_{H^*}$$ with a retraction of the inclusion, where $J_{H^*}$ is the radical of $H^*$. Furthermore, we have $$\operatorname{gr}^a(H^*)\cong A_{H^*}\times H^*/J_{H^*},$$
where $$A_{H^*}:=\{h\in \operatorname{gr}^a(H^*)\mid (id\otimes\tau)\Delta (h)=h\otimes 1\}.$$
\begin{lemma}\label{lem:A=R*ascoalgebra}
With the notations above, we have $R_H\cong (A_{H^*})^*$ as coalgebras.
\end{lemma}
\begin{proof}
We have
$$\operatorname{gr}^c(H)\cong (\operatorname{gr}^a(H^*))^*$$ as Hopf algebra. It follows from \cite[Theorem 5.1]{Mol77} that
$$R_H \times H_0\cong (A_{H^*}\times H^*/J_{H^*})^*\cong (A_{H*})^*\times H_0.$$
According to \cite[Theorem 3]{Rad85}, we know that $$R_H\cong (A_{H^*})^*$$ as coalgebras.
\end{proof}
In the following part, let $\operatorname{gr}^c(H)$ be a finite-dimensional Hopf algebra of tame corepresentation type.
Combining Lemma \ref{lemma:P^1=1^P} and Theorem \ref{thm:corepandquiver}, we know that one of the following three cases occurs:
   \begin{itemize}
  \item[(i)]$\mid\mathcal{P}{}^1\mid=1$ and $\mathcal{S}{}^1=\{C\}$ for some $C\in\mathcal{S}$ with $\dim_{\k}(C)=4$;
  \item[(ii)]$\mid\mathcal{P}{}^1\mid=2$ and $\mathcal{S}{}^1=\{\k g\}$ for some $g\in G(H)$;
  \item[(iii)]$\mid\mathcal{P}{}^1\mid=2$ and $\mathcal{S}{}^1=\{\k g, \k h\}$ for some $g, h\in G(H)$.
  \end{itemize}
  We need to determine which ideal in Theorem \ref{thm:tamestructure} makes $R_H\cong (\k\langle x,y\rangle/I)^*$ as coalgebras in each of these three cases.
We discuss them separately.

\subsection{Cases (i)}
Suppose $\mathcal{P}{}^1=\{\X\}$ and $\mathcal{S}{}^1=\{C\}$, where $$
\X=\left(\begin{array}{cc}
u\\
v
 \end{array}\right)
$$
 and $C$ is a $4$-dimensional simple subcoalgebra with basic multiplicative matrix
$$\C=\left(\begin{array}{cc}
c_{11}&c_{12}\\
c_{21}&c_{22}
 \end{array}\right).$$
By the definition of primitive matrix, we have
\begin{eqnarray*}
\Delta(u)&=&c_{11}\otimes u+c_{12}\otimes v+u\otimes 1,\\
\Delta(v)&=&c_{21}\otimes u+c_{22}\otimes v+v\otimes 1.
\end{eqnarray*}
It is clear that the subalgebra $U$ of $ \operatorname{gr}^c(H)$ generated by $u, v$ is contained in $R_{H}.$ We need to compute $\Delta_{R_H}(r)$ for any $r\in U.$

Before proceeding further, let us give the following lemma.
\begin{lemma}\label{lem:invertibleK}
With the notations above,
the set $\{c_{ij}u\mid 1\leq i,j\leq 2\}\cup\{c_{ij}v\mid 1\leq i,j\leq 2\}$ is linearly independent in $\operatorname{gr}^c(H)$. Moreover, there exists an invertible matrix $K=(k_{ij})_{4\times 4}$ over $\k$ such that
$
\C\odot^\prime\X=K(\X\odot\C),
$
namely,
$$
\left(\begin{array}{cccc}
c_{11}u&c_{12}u\\
c_{21}u&c_{22}u\\
c_{11}v&c_{12}v\\
c_{21}v&c_{22}v
 \end{array}\right)=
\left(\begin{array}{cccc}
k_{11}&k_{12}&k_{13}&k_{14}\\
k_{21}&k_{22}&k_{23}&k_{24}\\
k_{31}&k_{32}&k_{33}&k_{34}\\
k_{41}&k_{42}&k_{43}&k_{44}
 \end{array}\right)
 \left(\begin{array}{cccc}
uc_{11}& uc_{12}\\
uc_{21}&uc_{22}\\
vc_{11}&vc_{12}\\
vc_{21}&vc_{22}
 \end{array}\right).
$$
\end{lemma}
\begin{proof}
By \cite[Proposition 2.6]{Li22a}, there exists an invertible matrix $L$ over $\k$ such that
\begin{eqnarray*}
\left(\begin{array}{cc}
L\\
&I
 \end{array}\right)
 (\C\odot^{\prime}
\left(\begin{array}{cc}
\C&\X\\
&1
 \end{array}\right) )
 \left(\begin{array}{cc}
L^{-1}\\
&I
 \end{array}\right)
 =
 \left(\begin{array}{ccccc}
\D_1&&&\\
&\D_2&&&L(\C\odot^\prime\X)\\
&&\ddots\\
&&&\D_u\\
&&&&\C
 \end{array}\right),
\end{eqnarray*}
where $\D_1, \cdots, \D_u$ are the given basic multiplicative matrices.
Using \cite[Corollary 2.6 and Lemma 3.5]{YLL23}, we can show that the set $\{c_{ij}u\mid 1\leq i,j\leq 2\}\cup\{c_{ij}v\mid 1\leq i,j\leq 2\}$ is linearly independent in $\operatorname{gr}^c(H)$.
Let $$J=\left(\begin{array}{cccc}
1&0&0&0\\
0&0&1&0\\
0&1&0&0\\
0&0&0&1
 \end{array}\right)$$ be an invertible matrix over $\k$, we know that
 $$J(\C\odot^\prime\C)J^{-1}=\C\odot \C.$$
 It follows that
 \begin{eqnarray*}
\left(\begin{array}{cc}
LJ^{-1}\\
&I
 \end{array}\right)
(\left(\begin{array}{cc}
\C&\X\\
&1
 \end{array}\right) \odot\C)
 \left(\begin{array}{cc}
JL^{-1}\\
&I
 \end{array}\right)
 =
 \left(\begin{array}{ccccc}
\D_1&&&\\
&\D_2&&&LJ^{-1}(\X\odot\C)\\
&&\ddots\\
&&&\D_u\\
&&&&\C
 \end{array}\right).
\end{eqnarray*}
\begin{itemize}
  \item[(1)]
  Suppose $$C\cdot C=E^{(4)},$$ where $E^{(4)}\in\mathcal{S}$ is a $16$-dimensional simple subcoalgebra. We know that both $L(\C\odot^\prime\X)$ and $LJ^{-1}(\X\odot\C)$ are non-trivial $(\E^{(4)}, \C)$-primitive matrices, where $\E^{(4)}\in\mathcal{M}$ is the basic multiplicative matrix of $E^{(4)}$. From \cite[Corollary 2.16]{YLL23}, there exists an invertible matrix $P_1=\alpha I$ over $\k$ such that
  $$P_1(L(\C\odot^\prime\X))=LJ^{-1}(\X\odot\C).$$
  \item[(2)]Suppose $$C\cdot C=\k g+E^{(3)}$$ for some group-like element $g\in G(H)$ and some $9$-dimensional simple subcoalgebra $E^{(3)}\in\mathcal{S}$.
  According to \cite[Corollary 2.16]{YLL23}, there exists an invertible matrix
  $$P_2=\left(\begin{array}{cccc}
\alpha_1\\
&\alpha_2\\
&&\alpha_2\\
&&&\alpha_2
 \end{array}\right)$$ over $\k$ such that
  $$P_2(L(\C\odot^\prime\X))=LJ^{-1}(\X\odot\C).$$
  \item[(3)]Suppose $$C\cdot C=E_1^{(2)}+E_2^{(2)}$$ for some $4$-dimensional simple subcoalgebras $E_1^{(2)}, E_2^{(2)}\in\mathcal{S}$ and $E_1^{(2)}\neq E_2^{(2)}$. Using \cite[Corollary 2.16]{YLL23}, we obtain an invertible matrix
  $$P_3=\left(\begin{array}{cccc}
\alpha_1\\
&\alpha_1\\
&&\alpha_2\\
&&&\alpha_2
 \end{array}\right)$$ over $\k$ such that
  $$P_3(L(\C\odot^\prime\X))=LJ^{-1}(\X\odot\C).$$
  \item[(4)]
  Suppose $$C\cdot C=2E^{(2)}$$ for some $4$-dimensional simple subcoalgebra $E^{(2)}\in\mathcal{S}$. It follows from \cite[Proposition 2.15]{YLL23} that there exists an invertible matrix
  $$P_4=\left(\begin{array}{cccc}
\alpha_1&&\alpha_2&\\
&\alpha_1&&\alpha_2\\
\alpha_3&&\alpha_4&\\
&\alpha_3&&\alpha_4
 \end{array}\right)$$ over $\k$ such that
  $$P_4(L(\C\odot^\prime\X))=LJ^{-1}(\X\odot\C).$$
  \item[(5)]Suppose $$C\cdot C=\k g_1+\k g_2+\k g_3+\k g_4$$ for some group-like elements $g_1, g_2, g_3, g_4\in G(H)$. Note that $g_1, g_2, g_3, g_4$ are different with each other, otherwise the link quiver of $\operatorname{gr}^c(H)$ is not a Euclid diagram.
        By \cite[Corollary 2.16]{YLL23}, there exists an invertible matrix
  $$P_5=\left(\begin{array}{cccc}
\alpha_1\\
&\alpha_2\\
&&\alpha_3\\
&&&\alpha_4
 \end{array}\right)$$ over $\k$ such that
  $$P_5(L(\C\odot^\prime\X))=LJ^{-1}(\X\odot\C).$$
\end{itemize}
Based on the above argument, there exists some $1\leq i\leq 5$ such that invertible matrix $K=L^{-1}P_iLJ^{-1}$ over $\k$ satisfying $
\C\odot^\prime\X=K(\X\odot\C).
$
\end{proof}
In fact, for any $r\in U,$ $\Delta_{R_H}(r)$ is determined by the invertible matrix $K$ in Lemma \ref{lem:invertibleK}.
Next we consider case (i) under the assumption that $K$ is a diagonal matrix.
\begin{lemma}\label{lemma:InoI1}
Let $\operatorname{gr}^c(H)\cong (\k\langle x,y\rangle/I)^* \times H_0$ be a finite-dimensional coradically graded Hopf algebra over $\k$ of tame corepresentation type.
If $\mathcal{P}{}^1=\{\X\}$, $\mathcal{S}{}^1=\{C\}$ for some $C\in\mathcal{S}$ with $\dim_{\k}(C)=4$ and the invertible matrix $K$ in Lemma \ref{lem:invertibleK} is diagonal, namely
$$
K= \left(\begin{array}{cccc}
\alpha_1\\
&\alpha_2\\
&&\alpha_3\\
&&&\alpha_4
 \end{array}\right).
$$
Then $I\neq (x^2-y^2, yx-ax^2, xy)$, where $0\neq a\in\k$.
\end{lemma}
\begin{proof}
According to Lemma \ref{lem:deltaRH}, we have
\begin{eqnarray}\label{deltauv}
\Delta_{R_H}(uv)=1\otimes uv+uv\otimes 1+\alpha_3 v\otimes u+u\otimes v,
\end{eqnarray}
\begin{eqnarray}\label{deltavu}
\Delta_{R_H}(vu)=1\otimes vu+vu\otimes 1+\alpha_2 u\otimes v+v\otimes u,
\end{eqnarray}
\begin{eqnarray}\label{deltau2}
\Delta_{R_H}(u^2)=1\otimes u^2+u^2\otimes 1+(\alpha_1+1) u\otimes u,
\end{eqnarray}
\begin{eqnarray}\label{deltav2}
\Delta_{R_H}(v^2)=1\otimes v^2+v^2\otimes 1+(\alpha_4+1) v\otimes v.
\end{eqnarray}
If $\dim_\k(R_H)=\dim_{\k}((\k\langle x,y\rangle/(x^2-y^2, yx-ax^2, xy))^*)=4,$ then $u^2, v^2, uv, vu\in \k\{(x^2)^*\}.$
It follows that $\alpha_1=\alpha_4=-1,\;\;\alpha_2=\frac{1}{\alpha_3}.$
Then we have $u^2=v^2=0,\;\;uv=\alpha_2 vu.$
Hence $$R_{H}^*\cong \k\langle x,y\rangle/(x^2, y^2, xy-\alpha_2yx),$$ which indicates that $$I\neq (x^2-y^2, yx-ax^2, xy).$$
\end{proof}

\begin{lemma}\label{lemma:InoI3}
Let $\operatorname{gr}^c(H)\cong (\k\langle x,y\rangle/I)^* \times H_0$ be a finite-dimensional coradically graded Hopf algebra over $\k$ of tame corepresentation type.
Suppose $\mathcal{P}{}^1=\{\X\}$, $\mathcal{S}{}^1=\{C\}$ for some $C\in\mathcal{S}$ with $\dim_{\k}(C)=4$ and the invertible matrix $K$ in Lemma \ref{lem:invertibleK} is diagonal, namely
$$
K= \left(\begin{array}{cccc}
\alpha_1\\
&\alpha_2\\
&&\alpha_3\\
&&&\alpha_4
 \end{array}\right).
$$
If, in addition, $R_H$ is generated by $u, v$, then $I\neq (x^n-y^n, xy, yx)$, where $n\geq 2$.
\end{lemma}
\begin{proof}
 If $n=2$,
     using the same argument as in the proof of Lemma \ref{lemma:InoI1}, we can easily carry out the proof of this lemma. If $n\geq 3$, we know that $$(\k\langle x,y\rangle/I)^*(2)=\k\{(x^2)^*, (y^2)^*\}$$ and
\begin{eqnarray*}
  \Delta((x^2)^*)&=&(x^2)^*\otimes 1+ 1\otimes(x^2)^*+x^*\otimes x^*,\\
   \Delta((y^2)^*)&=&(y^2)^*\otimes 1+ 1\otimes(y^2)^*+y^*\otimes y^*.
  \end{eqnarray*}
Without loss of generality, suppose that
\begin{eqnarray}\label{u}
u&=&k_1x^*+k_2y^*,
\end{eqnarray}
\begin{eqnarray}\label{v}
v&=&k_3x^*+k_4y^*,
\end{eqnarray}
\begin{eqnarray}\label{u2}
u^2&=&k_5(x^2)^*+k_{6}(y^2)^*,
\end{eqnarray}
\begin{eqnarray}\label{v2}
v^2&=&k_7(x^2)^*+k_{8}(y^2)^*,
\end{eqnarray}
\begin{eqnarray}\label{uv}
uv&=&k_9(x^2)^*+k_{10}(y^2)^*,
\end{eqnarray}
\begin{eqnarray}\label{vu}
      vu&=&k_{11}(x^2)^*+k_{12}(y^2)^*,
\end{eqnarray}
where $k_i\in\k$ for $1\leq i \leq 12.$
      By substituting (\ref{u}-\ref{vu}) into (\ref{deltauv}-\ref{deltav2}), we obtain
       \begin{eqnarray*}
      (\alpha_1+1)k_1^2x^*\otimes x^*&=&k_5x^*\otimes x^*,\\
      (\alpha_1+1)k_1k_2x^*\otimes y^*&=&0,\\
      (\alpha_1+1)k_1k_2y^*\otimes x^*&=&0,\\
      (\alpha_1+1)k_2^2y^*\otimes y^*&=&k_{6}y^*\otimes y^*,\\
      (\alpha_4+1)k_3^2x^*\otimes x^*&=&k_7x^*\otimes x^*,\\
      (\alpha_4+1)k_3k_4x^*\otimes y^*&=&0,\\
      (\alpha_4+1)k_3k_4y^*\otimes x^*&=&0,\\
      (\alpha_4+1)k_4^2y^*\otimes y^*&=&k_{8}y^*\otimes y^*.\\
       (\alpha_2+1)k_1k_3x^*\otimes x^*&=&k_9x^*\otimes x^*,\\
      (\alpha_2k_1k_4+k_2k_3)x^*\otimes y^*&=&0,\\
      (\alpha_2k_2k_3+k_1k_4)y^*\otimes x^*&=&0,\\
      (\alpha_2+1)k_2k_4y^*\otimes y^*&=&k_{10}y^*\otimes y^*,\\
      (\alpha_3+1)k_1k_3x^*\otimes x^*&=&k_{11}x^*\otimes x^*,\\
      (\alpha_3k_2k_3+k_1k_4)x^*\otimes y^*&=&0,\\
      (\alpha_3k_1k_4+k_2k_3)y^*\otimes x^*&=&0,\\
      (\alpha_3+1)k_2k_4y^*\otimes y^*&=&k_{12}y^*\otimes y^*.
      \end{eqnarray*}
      Comparing the coefficients of the both side, we have
      \begin{eqnarray}\label{alpha1+1)k1k2=0}
      (\alpha_1+1)k_1k_2=0.
      \end{eqnarray}
      If $k_1=0,$ since $$(\alpha_3k_1k_4+k_2k_3)=0,$$ it follows that $k_2=0$ or $k_3=0,$  which is in contradiction with the fact that $u$ and $v$ are linearly independent.
     A similar argument shows that  $k_i\neq0$ for $1\leq i\leq 4.$ It follows from (\ref{alpha1+1)k1k2=0}) that $\alpha_1=-1.$ Moreover, because of the fact that$$(\alpha_4+1)k_3k_4=0,$$  we obtain $\alpha_4=-1.$
      This indicates that $u^2=v^2=0.$
      We claim that $\alpha_2\neq -1.$ Otherwise $k_9=k_{10}=0.$ Hence $uv=0,$ a contradiction.
      Note that $$\alpha_2(\alpha_2k_2k_3+k_1k_4)-(\alpha_2k_1k_4+k_2k_3)=0,$$
     direct computations shows that
      $\alpha_2=1.$ Using the same argument, we can obtain $\alpha_3=1.$
      Thus we have $uv=vu,$ which is a contradiction to $\dim_{\k}(R_H(2))=2$. The proof is complete.
\end{proof}

\begin{lemma}\label{lemma:InoI4}
Let $\operatorname{gr}^c(H)\cong (\k\langle x,y\rangle/I)^* \times H_0$ be a finite-dimensional coradically graded Hopf algebra over $\k$ of tame corepresentation type.
Suppose $\mathcal{P}{}^1=\{\X\}$, $\mathcal{S}{}^1=\{C\}$ for some $C\in\mathcal{S}$ with $\dim_{\k}(C)=4$ and the invertible matrix $K$ in Lemma \ref{lem:invertibleK} is diagonal, namely
$$
K= \left(\begin{array}{cccc}
\alpha_1\\
&\alpha_2\\
&&\alpha_3\\
&&&\alpha_4
 \end{array}\right).
$$
If, in addition, $R_H$ is generated by $u, v$, then $I\neq (x^2, y^2, (xy)^mx-(yx)^my)$, where $m\geq1$.
\end{lemma}
\begin{proof}
Suppose that
\begin{eqnarray*}
u&=&k_1x^*+k_2y^*,\\
v&=&k_3x^*+k_4y^*,\\
u^2&=&k_5(x^2)^*+k_{6}(y^2)^*,\\
v^2&=&k_7(x^2)^*+k_{8}(y^2)^*,
\end{eqnarray*}
where $k_i\in\k$ for $1\leq i \leq 8.$
Similar to the proof of Lemma \ref{lemma:InoI3}, we have
\begin{eqnarray*}
(\alpha_1+1)k_1^2&=&0,\\
(\alpha_1+1)k_1k_2&=&k_5,\\
(\alpha_1+1)k_1k_2&=&k_6,\\
(\alpha_1+1)k_2^2&=&0,\\
(\alpha_4+1)k_3^2&=&0,\\
(\alpha_4+1)k_3k_4&=&k_7,\\
(\alpha_4+1)k_3k_4&=&k_8,\\
(\alpha_4+1)k_4^2&=&0.
\end{eqnarray*}
It is straightforward to show that $\alpha_1=\alpha_4=-1$ and thus $u^2=v^2=0.$
Since $(uv)^mu, (vu)^mv\in\k\{((xy)^mx)^*\}$, it follows that
\begin{eqnarray}\label{k9}
(uv)^mu=k_9(vu)^mv
\end{eqnarray}
 for some $k_9\in\k.$
Note that
\begin{eqnarray*}
\Delta((uv)^mu)&=&(\Delta(uv))^m\Delta(u)\\
&=&(c_{11}c_{21}\otimes u^2+c_{11}c_{22}\otimes uv+c_{12}c_{21}\otimes vu+c_{12}c_{22}\otimes v^2+uv\otimes1\\
&&+c_{11}v\otimes u+c_{12}v\otimes v+uc_{21}\otimes u+uc_{22}\otimes v)^m(c_{11}\otimes u+c_{12}\otimes v+u\otimes 1),\\
\Delta((vu)^mv)&=&(\Delta(vu))^m\Delta(v)\\
&=&(c_{21}c_{11}\otimes u^2+c_{21}c_{12}\otimes uv+c_{22}c_{11}\otimes vu+c_{22}c_{12}\otimes v^2+vu\otimes1\\
&&+c_{21}u\otimes u+c_{22}u\otimes v+vc_{11}\otimes u+vc_{12}\otimes v)^m(c_{21}\otimes u+c_{22}\otimes v+v\otimes 1).
\end{eqnarray*}
It follows from (\ref{k9}) that $$((uv)^mc_{11}+c_{11}(vu)^m)\otimes u=k_9 ((vu)^mc_{21}+c_{21}(uv)^m)\otimes u.$$
This means that $$((\Pi_R\circ\eta^{-1}) \otimes id)(((uv)^mc_{11}+c_{11}(vu)^m)\otimes u)=k_9((\Pi_R\circ\eta^{-1}) \otimes id)(((vu)^mc_{21}+c_{21}(uv)^m)\otimes u).$$
It turns out that $$((uv)^m+(-1)^m\alpha_3^{m}(vu)^m)=0.$$ This contradicts the fact that $R_H$ is generated by $u,v$ and $$\dim_{\k}(R_H(2m))=\dim_{\k}((\k\langle x,y\rangle/I_2)(2m))=2.$$
Thus $$I\neq (x^2, y^2, (xy)^mx-(yx)^my),$$ where $m\geq1$.
\end{proof}

For our purpose, we need to consider the following combinatorial functors:
\begin{eqnarray*}
H_1(m, l, t)&=&\sum\limits_{0\leq m_1\leq m_2\leq\cdots\leq m_l\leq m-l}t^{\sum_{i=1}^lm_i},\\
H_2(m, l, t)&=&\sum\limits_{0\leq n_1+ n_2+\cdots+ n_l\leq m-l}t^{\sum_{i=1}^l(l+1-i)n_i},\\
H_3(m, l, t)&=&t^{m-l}\sum\limits_{0\leq n_1+ n_2+\cdots+ n_{l-1}\leq m-l}t^{\sum_{i=1}^{l-l}(l-i)n_i}+\sum\limits_{0\leq n_1+ n_2+\cdots+ n_l\leq m-l}t^{\sum_{i=1}^l(l+1-i)n_i}.
\end{eqnarray*}
Here $m, l\in\Bbb{Z}_+, 0< l< m, m_1, \cdots, m_l, n_1, \cdots, n_l\in\Bbb{N}$ and $t$ is an indeterminant.
\begin{lemma}\emph{(}\cite[Lemma 3.1, Proposition 3.2]{HL09}\emph{)}
We have
\begin{itemize}\label{lem:H1=H2=H3}
  \item[(1)]$H_1(m, l, t)=H_2(m, l, t)=H_3(m, l, t);$
  \item[(2)]$H_1(m, l, t)=0$ for all $0< l< m$ if and only if $t$ is an $m$th primitive root of unit.
\end{itemize}
\end{lemma}
With the help of the preceding lemmas, we can get the main result for case (i).
\begin{proposition}\label{prop:I=I2}
Let $\operatorname{gr}^c(H)\cong (\k\langle x,y\rangle/I)^* \times H_0$ be a finite-dimensional coradically graded Hopf algebra over $\k$ of tame corepresentation type.
Suppose $\mathcal{P}{}^1=\{\X\}$, $\mathcal{S}{}^1=\{C\}$ for some $C\in\mathcal{S}$ with $\dim_{\k}(C)=4$ and the invertible matrix $K$ in Lemma \ref{lem:invertibleK} is diagonal, namely
$$
K= \left(\begin{array}{cccc}
\alpha_1\\
&\alpha_2\\
&&\alpha_3\\
&&&\alpha_4
 \end{array}\right).
$$
If, in addition, $R_H$ is generated by $u, v$, then
\begin{itemize}
  \item[(1)]$I=(x^2, y^2, (xy)^m-a(yx)^m)$ for $0\neq a\in\k$ and $m\geq1$;
  \item[(2)]$\alpha_1=\alpha_4=-1$;
  \item[(3)]$a=(-1)^{m-1}\alpha_2^m$ or $a=(-1)^{m-1}\alpha_3^m$;
  \item[(4)]$\alpha_2\alpha_3$ is an $m$th primitive root of unity.
\end{itemize}
\end{proposition}
\begin{proof}
\begin{itemize}
  \item[(1)]Combining Theorem \ref{thm:tamestructure}, Lemmas \ref{lemma:InoI1}, \ref{lemma:InoI3} and \ref{lemma:InoI4}, we know that $$I=(x^2, y^2, (xy)^m-a(yx)^m)$$ for $0\neq a\in\k$ and $m\geq1$.
  \item[(2)]An argument similar to the one used in the proof of Lemma \ref{lemma:InoI4} shows that $\alpha_1=\alpha_4=-1$ and $u^2=v^2=0.$
  \item[(3)]
Note that \begin{eqnarray*}
\Delta((uv)^m)&=&(\Delta(uv))^m\\
&=&(c_{11}c_{21}\otimes u^2+c_{11}c_{22}\otimes uv+c_{12}c_{21}\otimes vu+c_{12}c_{22}\otimes v^2+uv\otimes1\\
&&+c_{11}v\otimes u+c_{12}v\otimes v+uc_{21}\otimes u+uc_{22}\otimes v)^m,\\
\Delta((vu)^m)&=&(\Delta(vu))^m\\
&=&(c_{21}c_{11}\otimes u^2+c_{21}c_{12}\otimes uv+c_{22}c_{11}\otimes vu+c_{22}c_{12}\otimes v^2+vu\otimes1\\
&&+c_{21}u\otimes u+c_{22}u\otimes v+vc_{11}\otimes u+vc_{12}\otimes v)^m.
\end{eqnarray*}
Besides, in $(\k\langle x,y\rangle/(x^2, y^2, (xy)^m-a(yx)^m))^*$, we have
  \begin{eqnarray*}\Delta(((xy)^m)^*)&=&1\otimes ((xy)^m)^*+x^*\otimes (y(xy)^{m-1})^*+(xy)^*\otimes((xy)^{m-1})^*\\
  &+&\cdots+((xy)^i)^*\otimes ((xy)^{m-i})^*+((xy)^ix)^*\otimes (y(xy)^{m-1-i})^*\\
  &+&+\cdots+((xy)^{m-1}x)^*\otimes y^*+((xy)^m)^*\otimes1\\
  &+&\frac{1}{a}(1\otimes ((yx)^m)^*+y^*\otimes (x(yx)^{m-1})^*+(yx)^*\otimes((yx)^{m-1})^*\\
  &+&\cdots+((yx)^i)^*\otimes ((yx)^{m-i})^*+((yx)^iy)^*\otimes (x(yx)^{m-1-i})^*\\
  &+&+\cdots+((yx)^{m-1}y)^*\otimes x^*+((yx)^m)^*\otimes1).
  \end{eqnarray*}
  Suppose that
  \begin{eqnarray}\label{uprime}
  u=k_1x^*+k_2y^*,
  \end{eqnarray}
  \begin{eqnarray}\label{vprime}
  v=k_3x^*+k_4y^*,
  \end{eqnarray}
  \begin{eqnarray}\label{uvprime}
  (vu)^m=k_5(((xy)^m)^*),
  \end{eqnarray}
  where $k_i\in\k$ for $1\leq i\leq 5.$
  By substituting (\ref{uprime}) and (\ref{vprime}) into (\ref{uvprime}), we obtain
  \begin{eqnarray*}
  (k_1(vu)^{m-1}v+k_3(\alpha_3)^m(-1)^{m-1}u(vu)^{m-1})\otimes x^*&=&k_5\frac{1}{a}((yx)^{m-1}y)^*\otimes x^*,\\
  x^*\otimes (k_1(\alpha_3)^m(-1)^{m-1}(vu)^{m-1}v+k_3u(vu)^{m-1})&=&x^*\otimes k_5((yx)^{m-1}y)^*.
  \end{eqnarray*}
  It follows that $$k_1(\alpha_3)^m(-1)^{m-1}=\frac{1}{a}k_1$$ and $$k_3=k_3(\alpha_3)^m(-1)^{m-1}\frac{1}{a}.$$
  If $k_1=0$ and $k_3\neq0$, then $a=(-1)^{m-1}(\alpha_3)^m.$ If $k_1\neq0$ and $k_3=0$, then $a=(-1)^{m-1}(\alpha_2)^m.$
  If $k_1\neq0$ and $k_3\neq0$, then $a=(-1)^{m-1}(\alpha_3)^m=(-1)^{m-1}(\alpha_2)^m.$
  \item[(4)]We shall adopt the same procedure as in the proof of Lemma \ref{lemma:InoI4}. Suppose that $$(uv)^m=k_6(vu)^m,$$
for some $k_6\in\k$.
It follows from $$\Delta((uv)^m)=k_6\Delta((vu)^m)$$
 that
$$((uv)^{m-1}uc_{21}+c_{11}(vu)^{m-1}v)\otimes u=k_6((vu)^{m-1}vc_{11}+c_{21}(uv)^{m-1}v)\otimes u$$
and
$$((uv)^{m-1}uc_{22}+c_{12}(vu)^{m-1}v)\otimes v=k_6((vu)^{m-1}vc_{12}+c_{22}(uv)^{m-1}u)\otimes v.$$
Thus we have
\begin{eqnarray*}
&&((\Pi_R\circ\eta^{-1}) \otimes id)(((uv)^{m-1}uc_{21}+c_{11}(vu)^{m-1}v)\otimes u)\\
&=&k_0((\Pi_R\circ\eta^{-1}) \otimes id)(k_6((vu)^{m-1}vc_{11}+c_{21}(uv)^{m-1}v)\otimes u),
\end{eqnarray*}
and
\begin{eqnarray*}
&&((\Pi_R\circ\eta^{-1}) \otimes id)(((uv)^{m-1}uc_{22}+c_{12}(vu)^{m-1}v)\otimes v)\otimes u)\\
&=&k_0((\Pi_R\circ\eta^{-1}) \otimes id)(k_6((vu)^{m-1}vc_{12}+c_{22}(uv)^{m-1}u)\otimes v).
\end{eqnarray*}
Direct computations shows that
\begin{eqnarray*}
(-1)^{m-1}\alpha_3^{m}=k_6,
\end{eqnarray*}
\begin{eqnarray*}
1=k_{6}(-1)^{m-1}\alpha_2^{m}.
\end{eqnarray*}
It follows that $(\alpha_2\alpha_3)^m=1.$
Note that for any element $f(u, v)$ generated by $u, v$, we can always write uniquely $\Delta(f(u, v))$ in the following form:
\begin{eqnarray*}
&&f(u, v)\otimes 1+(f(u, v))_u\otimes u+(f(u, v))_v\otimes v+(f(u, v))_{uv}\otimes uv+\cdots\\
&+&(f(u, v))_{(uv)^i}\otimes (uv)^i+ (f(u, v))_{vu}^i\otimes (vu)^i+(f(u, v))_{(uv)^iu}\otimes (uv)^iu\\
&+&(f(u, v))_{(vu)^iv}\otimes (vu)^iv+\cdots .
\end{eqnarray*}
Since $$(uv)^m=(-1)^{m-1}\alpha_3^{m}(vu)^m,$$
it follows that $$(\Pi_R\circ\eta^{-1} \otimes id)\Delta((uv)^m)=(\Pi_R\circ\eta^{-1} \otimes id)\Delta((-1)^{m-1}\alpha_3^{m}(vu)^m).$$
But $\varepsilon(c_{12})=\varepsilon(c_{21})=0,$ we only need to focus on
$$(c_{11}c_{22}\otimes uv+uv\otimes1
+c_{11}v\otimes u+uc_{22}\otimes v)^m$$
and
$$
(c_{22}c_{11}\otimes vu+vu\otimes1
+c_{22}u\otimes v+vc_{11}\otimes u)^m.
$$
Note that for any $0< l< m$, $u$ and $v$ should appear alternately in the left items in $(uv)^m_{(uv)^l}$. By this observation, the items starting with $u$ in $(uv)^m_{(uv)^l}$ are just
$$\sum\limits_{0\leq n_1+ n_2+\cdots +n_{l}\leq m-l}(uv)^{n_{1}}c_{11}c_{22}(uv)^{n_{2}}c_{11}c_{22}\cdots c_{11}c_{22}(uv)^{n_{l}}c_{11}c_{22}(uv)^{n_{l+1}}.$$
But the items starting with $u$ in $(vu)^m_{(vu)^l}$ is 0.
This indicates that
\begin{eqnarray*}
&&\sum\limits_{0\leq n_1+ n_2+\cdots +n_{l}\leq m-l}(uv)^{n_{1}}c_{11}c_{22}(uv)^{n_{2}}c_{11}c_{22}\cdots c_{11}c_{22}(uv)^{n_{l}}c_{11}c_{22}(uv)^{n_{l+1}}\\
&=&\sum\limits_{0\leq n_1+ n_2+\cdots +n_{l}\leq m-l}(\alpha_2\alpha_3)^{n_1}(\alpha_2\alpha_3)^{n_1+n_2}\cdots(\alpha_2\alpha_3)^{n_1+n_2+\cdots +n_l}(c_{11}c_{22})^l\\&&(uv)^{m-l}\\
&=&H_2(m, l, \alpha_2\alpha_3)(c_{11}c_{22})^l(uv)^{m-l}\\
&=&0.
\end{eqnarray*}
Using Lemma \ref{lem:H1=H2=H3}, we know that $\alpha_2\alpha_3$ is an $m$th primitive root of unity.
\end{itemize}
\end{proof}
\begin{corollary}
With the notations in Proposition \ref{prop:I=I2}, if $m\geq2$, then $$c_{11}c_{12}=c_{12}c_{11}=c_{21}c_{22}=c_{22}c_{21}=0.$$
\end{corollary}
\begin{proof}
According to the proof of Proposition \ref{prop:I=I2}, we know that $u^2=v^2=0.$
This means that $\Delta(u^2)=\Delta(v^2)=0.$
Since $m\geq2$, it follows that $uv, vu$ are linearly independent. Thus we have $$c_{11}c_{12}=c_{12}c_{11}=c_{21}c_{22}=c_{22}c_{21}=0.$$
\end{proof}

To conclude this subsection, we have considered case (i) under the assumption that $K$ in Lemma \ref{lem:invertibleK} is a diagonal matrix. Without this assumption, it remains an open problem which ideal in Theorem \ref{thm:tamestructure} occurs. However, if $K$ is given, the same method can be applied.
\subsection{Cases (ii) and (iii)}
\begin{proposition}\label{prop:g=g,gnoh}
Let $\operatorname{gr}^c(H)\cong (\k\langle x,y\rangle/I)^* \times H_0$ be a finite-dimensional coradically graded Hopf algebra over $\k$ of tame corepresentation type.
\begin{itemize}
  \item[(1)]If $\mid\mathcal{P}{}^1\mid=2$ and $\mathcal{S}{}^1=\{\k g\}$ for some $g\in G(H)$, then $I=(x^2, y^2, xy+yx)$;
  \item[(2)]If $\mid\mathcal{P}{}^1\mid=2$ and $\mathcal{S}{}^1=\{\k g, \k h\}$ for some $g, h\in G(H)$, then $I=(x^2, y^2, (xy)^m-a(yx)^m)$.
\end{itemize}
\end{proposition}
\begin{proof}
It follows from Proposition \ref{HtameiffH0tame} that the link-indecomposable component $(\operatorname{gr}(H))_{(1)}$ containing $\k 1$ is of tame corepresentation type.
According to \cite[Proposition 4.14]{YLL23}, in cases (ii) and (iii), $(\operatorname{gr}(H))_{(1)}$ is a pointed Hopf algebra. Therefore, the desired conclusion follows from \cite[Theorems 4.9 and 4.16]{HL09}.
\end{proof}
Indeed, Proposition \ref{prop:g=g,gnoh} can also be obtained by the same reasoning as in the proofs of Lemmas \ref{lemma:InoI1}, \ref{lemma:InoI3}, \ref{lemma:InoI4} and Proposition \ref{prop:I=I2}.
Moreover, using the same argument, we can easily prove the following remark.
\begin{remark}\label{remark:g=g,gnoh}\rm
Let $\operatorname{gr}^c(H)\cong (\k\langle x,y\rangle/I)^* \times H_0$ be a finite-dimensional coradically graded Hopf algebra of tame corepresentation type.
\begin{itemize}
  \item[(1)]If $\mid\mathcal{P}{}^1\mid=2$ and $\mathcal{S}{}^1=\{\k g\}$ for some $g\in G(H)$, suppose that
$$gu=\alpha_1ug+\alpha_2vg,\;\;gv=\alpha_3ug+ \alpha_4 vg$$for some $\alpha_1, \alpha_2, \alpha_3, \alpha_4\in\k.$ Then $\alpha_1=\alpha_4=-1, \alpha_2=\alpha_3=0;$
  \item[(2)]If $\mid\mathcal{P}{}^1\mid=2$ and $\mathcal{S}{}^1=\{\k g, \k h\}$ for some $g, h\in G(H)$, assume that $$gu=\beta_1ug,\; gv=\beta_2vg,\;\;hu=\beta_3uh,\;hv=\beta_4vh$$ for some $\beta_1, \beta_2, \beta_3, \beta_4\in\k.$ Then \begin{itemize}
  \item [(i)]$\beta_1=\beta_4=-1$;
\item [(ii)]$a$ in Proposition \ref{prop:g=g,gnoh} equals $(-1)^{m-1}\beta_2$ or $(-1)^{m-1}\beta_3$;
\item [(iii)]$\beta_2\beta_3$ is an $m$th primitive root of unity.
\end{itemize}
\end{itemize}
\end{remark}
It should be pointed out that the above remark coincides with \cite[Lemma 4.8, Proposition 4.15]{HL09}.

\section{Examples}\label{section7}
In this section, we give several examples of finite-dimensional coradically graded link-indecomposable Hopf algebras over $\k$ with the dual Chevalley property of tame corepresentation type in the following three cases:
 \begin{itemize}
  \item[(i)]$\mid\mathcal{P}{}^1\mid=1$ and $\mathcal{S}{}^1=\{C\}$ for some $C\in\mathcal{S}$ with $\dim_{\k}(C)=4$;
  \item[(ii)]$\mid\mathcal{P}{}^1\mid=2$ and $\mathcal{S}{}^1=\{\k g\}$ for some $g\in G(H)$;
  \item[(iii)]$\mid\mathcal{P}{}^1\mid=2$ and $\mathcal{S}{}^1=\{\k g, \k h\}$ for some $g, h\in G(H)$.
  \end{itemize}

In fact, if $H$ is link-indecomposable, it follows from Lemma \ref{lem:generate} that
the coradical of $H$ is generated by $\{\span(C)\mid C\in {}^1 \mathcal{S}\}\cup\{\span(S(C))\mid C\in {}^1 \mathcal{S}\}.$ In particular, combining \cite[Lemma 2.1]{HL09} and \cite[Proposition 4.14]{YLL23}, we know that $(H_{(1)})_0$ is an abelian group in cases (ii) and (iii).
 According to \cite[Remark 4.10]{HL09}, we have the following lemma.
\begin{lemma}\label{prop:g=gexample}
Let $H$ be the algebra which is generated by $g, u, v$ satisfying the following relations:
$$gu=-ug,\;\;gv=-vg,\;\;uv=-vu,\;\;u^2=v^2=0,\;\;g^n=1,$$
where $n$ is an even number.
Moreover, the coalgebra structure and antipode are given by:
$$\Delta(g)=g\otimes g,\;\;\varepsilon(g)=1,\;\;S(g)=g^{-1},$$
$$\Delta(u)=g\otimes u+u\otimes 1,\;\;\varepsilon(u)=0,\;\;S(u)=-g^{-1}u,$$
$$\Delta(v)=g\otimes v+v\otimes 1,\;\;\varepsilon(v)=0,\;\;S(v)=-g^{-1}v.$$
Then $H$ is a coradically graded Hopf algebra of tame corepresentation type with $\mid\mathcal{P}{}^1\mid=2$ and $\mathcal{S}{}^1=\{\k g\}$.
 Moreover, we have $$H\cong (\k\langle x,y\rangle/(x^2, y^2, xy+yx))^*\times \k\langle g\rangle.$$
\end{lemma}

From \cite[Remark 4.17(2)]{HL09}, we have the following example.
\begin{example}\label{prop:gnohexample}\rm
Let $H$ be the algebra which is generated by $g, h, u, v$ satisfying the following relations:
$$gh=hg,\;\; g^{n_1}=h^{n_2}=1,$$
$$gu=-ug,\;\; gv=\alpha vg,\;\;hu=\beta uh,\;\;hv=-vh,$$
$$u^2=v^2=0, (uv)^m=(-1)^{m-1}\beta^m(vu)^m,$$
where $n_1, n_2\in\Bbb{Z},$ $\alpha\beta$ is an $m$th primitive root of unit and $m\mid l.c.m(n_1, n_2)$.
The coalgebra structure and antipode are given by
$$\Delta(g)=g\otimes g,\;\;\varepsilon(g)=1,\;\;S(g)=g^{-1},$$
$$\Delta(h)=h\otimes h,\;\;\varepsilon(h)=1,\;\;S(h)=h^{-1},$$
$$\Delta(u)=g\otimes u+u\otimes 1,\;\;\varepsilon(u)=0,\;\;S(u)=-g^{-1}u,$$
$$\Delta(v)=h\otimes v+v\otimes 1,\;\;\varepsilon(v)=0,\;\;S(v)=-h^{-1}v.$$
Then $H$ is a coradically graded Hopf algebra of tame corepresentation type with $\mid\mathcal{P}{}^1\mid=2$ and $\mathcal{S}{}^1=\{\k g, \k h\}$.
 Moreover, we have $$H\cong (\k\langle x,y\rangle/(x^2, y^2, (xy)^m-(-1)^{m-1}\beta^m(yx)^m))^*\times \k\langle g, h\rangle.$$
\end{example}

In cases (ii) and (iii), according to Proposition \ref{prop:g=g,gnoh} and Remark \ref{remark:g=g,gnoh}, only certain special ideals of the form $\{(x^2, y^2, (xy)^m-a(yx)^m)\mid 0\neq a\in\k, m\geq 1\}$ can appear. When such an ideal occurs, we can construct coradically graded Hopf algebras of tame corepresentation type over $\k G$ for some $G=G(H)$ as in Examples \ref{prop:g=gexample} and \ref{prop:gnohexample}. However, in case (i), we do not yet know how to find all $H^\prime$ such that $(\k\langle x,y\rangle/I)^* \times H^\prime$ is a Hopf algebra for some special ideals $I$ listed in Theorem \ref{thm:tamestructure}, even when the invertible matrix $K$ in Lemma \ref{lem:invertibleK} is diagonal.

In the following, we give some examples of link-indecomposable coradically graded Hopf algebras of tame corepresentation type over 8-dimensional non-pointed cosemisimple Hopf algebras, where the invertible matrix $K$ in Lemma \ref{lem:invertibleK} is diagonal.

According to \cite[Theorem 2. 13]{Mas95}, we have the following lemma.
\begin{lemma}
Non-pointed $8$-dimensional semisimple Hopf algebras over $\k$ consist of $3$ isomorphic classes, represented by
$$(\k D_8)^*, \;\; (\k Q_8)^*\;\;, H_8,$$
where $D_8=\langle x, y\mid x^4=y^2=1, yx=x^{-1}y \rangle$ is the dihedral group and $Q_8=\langle x, y\mid x^4=1, y^2=x^2, yx=x^{-1}y\rangle$ is the quaternion group. Among these, $H_8$ is the unique one that is neither commutative nor cocommutative, and is generated as an algebra by $x, y, z$ with relations
\begin{eqnarray}
x^2=y^2=1, \;z^2=\frac{1}{2}(1+x+y-xy),\;yx=xy,\;zx=yz,\;zy=xz;
\end{eqnarray}\label{relation1}
the coalgebra structure and antipode are given by:
\begin{eqnarray}\label{relation2}
\Delta(x)=x\otimes x,\; \Delta(y)=y\otimes y,\;\varepsilon(x)=\varepsilon(y)=1,
\end{eqnarray}
\begin{eqnarray}\label{relation3}
\Delta(z)=\frac{1}{2}(1\otimes 1+1\otimes x+ y\otimes 1-y\otimes x)(z\otimes z),\; \varepsilon(z)=1,
\end{eqnarray}
\begin{eqnarray}\label{relation4}
S(x)=x,\; S(y)=y,\; S(z)=z.
\end{eqnarray}
\end{lemma}
According to Lemma \ref{lem:generate}, when we consider link-indecomposable coradically graded Hopf algebras of tame corepresentation type over 8-dimensional non-pointed cosemisimple Hopf algebras, we only need to consider case (i).
\subsection{Hopf algebras of tame corepresentation type over $(\k D_8)^*$}

Let $\{e_{pq}\}_{p=0, 1, 2, 3; q=0, 1}$ be the basis of $(\k D_8)^*$, dual to the basis $\{x^py^q\}_{p=0, 1, 2, 3; q=0, 1}$ of $\k D_8$. The multiplication and unit are given, respectively,  by
\begin{eqnarray}\label{relationd81}
e_{p_1q_1}e_{p_2 q_2}=\delta_{p_1, p_2}\delta_{q_1, q_2}e_{p_1q_1},\;\;
1=\sum\limits_{p, q}e_{pq},
\end{eqnarray}
the coalgebra structure and antipode are given by
\begin{eqnarray}\label{comultid8}
\Delta(e_{pq})=\sum_{\stackrel{p_1+p_2+2q_1p_2\equiv\; p\; mod\; 4}{ q_1+q_2\equiv \;q\; mod\; 2}} e_{p_1q_1}\otimes e_{p_2q_2},\;\;\varepsilon(e_{pq})=\delta_{p, 0}\delta_{q, 0},
\end{eqnarray}
 \begin{eqnarray}\label{Sd8}
S(e_{pq})&=&e_{p^\prime q^\prime}, \text{where } p+p^\prime+2qp^\prime\equiv\; 0\;\; mod\;\; 4,\;\;  q+q^\prime\equiv \;0\;\; mod\;\; 2.
\end{eqnarray}
Clearly,
$X=\sum\limits_{pq}(-1)^pe_{pq},
Y=\sum\limits_{pq}(-1)^qe_{pq}$
are group-like elements of order 2. Let
 $$\C=\left(\begin{array}{cc}
c_{11}&c_{12}\\
c_{21}&c_{22}
\end{array}\right)=
 \left(\begin{array}{cc}
e_{00}-\sqrt{-1}e_{10}-e_{20}+\sqrt{-1}e_{30}&\sqrt{-1}e_{01}+e_{11}-\sqrt{-1}e_{21}-e_{31}\\
-\sqrt{-1}e_{01}+e_{11}+\sqrt{-1}e_{21}-e_{31}&e_{00}+\sqrt{-1}e_{10}-e_{20}-\sqrt{-1}e_{30}
\end{array}\right).$$ It is a basic multiplicative matrix of $C$, where $C=\operatorname{span}\{c_{11}, c_{12}, c_{21}, c_{22}\}$.
Thus the simple subcoalgebras in $(\k D_8)^*$ are $\k1, \k X, \k Y, \k XY,  C$.

Next we construct a link-indecomposable coradically graded Hopf algebra $H$ of tame corepresentaion type over $(\k D_8)^*$ such that the invertible matrix $K$ in Lemma \ref{lem:invertibleK} is diagonal. Namely,
suppose there exists an diagonal invertible matrix $K=(k_{ij})_{4\times 4}$ over $\k$ such that
$$
\C\odot^\prime\X=K(\X\odot\C),
$$
where $\X=
\left(\begin{array}{cc}
u\\
v
 \end{array}\right)$
 is a non-trivial $(\C, \k1)$-primitive matrix,
and $$
K=
\left(\begin{array}{cccc}
\alpha_1\\
&\alpha_2\\
&&\alpha_3\\
&&&\alpha_4
 \end{array}\right)
 .
 $$
According to Proposition \ref{prop:I=I2}, if $R_{H}=\{h\in H\mid (id\otimes\pi)\Delta (h)=h\otimes 1\}$ is generated by $u, v$, we know that $\alpha_1=\alpha_4=-1.$
Since $c_{11}c_{22}+c_{12}c_{21}=1, $ then
\begin{eqnarray*}
(c_{11}c_{22}+c_{12}c_{21})u=-\alpha_2u(c_{11}c_{22}+c_{12}c_{21})
=u(c_{11}c_{22}+c_{12}c_{21}).
\end{eqnarray*}
It follows that $\alpha_2=-1.$
Next we consider
$(c_{11}c_{22}+c_{12}c_{21})v,$
 a similar argument shows that $\alpha_3=-1.$
As a summary, we have
\begin{example}\label{exampleD8}\rm
Let $H$ be a Hopf algebra generated as an algebra by $\{c_{ij}\}_{i=1, 2; j=1, 2}, u, v$ satisfying (\ref{relationd81}) and the following relations:
$$
\C\odot^\prime\X=K(\X\odot\C),
\;\;u^2=v^2=0,\;\; uv+vu=0,$$
where
$$
K=
\left(\begin{array}{cccc}
-1\\
&-1\\
&&-1\\
&&&-1
 \end{array}\right)
 .
 $$
 The coalgebra structure and antipode are given by (\ref{comultid8}), (\ref{Sd8}) and
 $$
 \Delta(u)=c_{11}\otimes u+c_{12}\otimes v+u\otimes 1,\;\;\varepsilon(u)=0,
 $$
 $$
\Delta(v)=c_{21}\otimes u+c_{22}\otimes v+v\otimes 1,\;\;\varepsilon(v)=0,
$$
$$
S(u)=-(e_{00}-\sqrt{-1}e_{30}-e_{20}-\sqrt{-1}e_{10})u-(\sqrt{-1}e_{01}+e_{11}-\sqrt{-1}e_{21}-e_{31})v,
$$
$$
S(v)=-(-\sqrt{-1}e_{01}+e_{11}+\sqrt{-1}e_{21}-e_{31})u-(e_{00}+\sqrt{-1}e_{30}-e_{20}-\sqrt{-1}e_{10})v.
$$
One can show that $H\cong(\k\langle x,y\rangle/(x^2, y^2, (xy)^2+(yx)^2))^*\times (\k D_8)^*$, and it is a link-indecomposable coradically graded Hopf algebra of tame corepresentaion type over $(\k D_8)^*$.
\end{example}

\subsection{Hopf algebras of tame corepresentation type over $(\k Q_8)^*$}
Let $\{e_{pq}\}_{p=0, 1, 2, 3; q=0, 1}$ be the basis of $(\k Q_8)^*$, dual to the basis $\{x^py^q\}_{p=0, 1, 2, 3; q=0, 1}$ of $\k Q_8$. The multiplication and unit are given, respectively,  by
\begin{eqnarray}\label{relationQ8}
e_{p_1q_1}e_{p_2 q_2}=\delta_{p_1, p_2}\delta_{q_1, q_2}e_{p_1q_1},\;\;
1=\sum\limits_{p, q}e_{pq},
\end{eqnarray}
the coalgebra structure and antipode are given by
\begin{eqnarray}\label{comultiQ8}
\Delta(e_{pq})=\sum_{\stackrel{p_1+p_2+2q_1(p_2+q_2)\equiv\; p\; mod\; 4}{ q_1+q_2\equiv \;q\; mod\; 2}} e_{p_1q_1}\otimes e_{p_2q_2}, \;\;\varepsilon(e_{pq})=\delta_{p, 0}\delta_{q, 0},
\end{eqnarray}
 \begin{eqnarray}\label{SQ8}
S(e_{pq})&=&e_{p^\prime q^\prime}, \text{where } p+p^\prime+2q(p^\prime+q^\prime)\equiv\; 0\;\;mod\;\; 4,\;\; q+q^\prime\equiv \;0\;\; mod\;\; 2.
\end{eqnarray}
It is easy to check that elements
$X=\sum\limits_{pq}(-1)^pe_{pq},
Y=\sum\limits_{pq}(-1)^qe_{pq}$
are group-like elements of order 2. Let
$$\C=\left(\begin{array}{cc}
c_{11}&c_{12}\\
c_{21}&c_{22}
\end{array}\right)=
 \left(\begin{array}{cc}
e_{00}+\sqrt{-1}e_{01}-e_{20}-\sqrt{-1}e_{21}&\sqrt{-1}e_{10}+e_{11}-\sqrt{-1}e_{30}-e_{31}\\
\sqrt{-1}e_{01}-e_{11}-\sqrt{-1}e_{30}+e_{31}&e_{00}-\sqrt{-1}e_{01}-e_{20}+\sqrt{-1}e_{21}
\end{array}\right).$$
Then $\C$ is a basic multiplicative matrix of $C$, where $C=\operatorname{span}\{c_{11}, c_{12}, c_{21}, c_{22}\}$.
Thus the simple subcoalgebras in $(\k Q_8)^*$ are $\k1, \k X, \k Y, \k XY,  C$.

Suppose that
there exists an diagonal invertible matrix $K=(k_{ij})_{4\times 4}$ over $\k$ such that
$$
\C\odot^\prime\X=K(\X\odot\C),
$$
where $\X=
\left(\begin{array}{cc}
u\\
v
 \end{array}\right)$
 is a non-trivial $(\C, \k1)$-primitive matrix, and
$$
K=
\left(\begin{array}{cccc}
\alpha_1\\
&\alpha_2\\
&&\alpha_3\\
&&&\alpha_4
 \end{array}\right)
 .
 $$
 Suppose that $R_{H}=\{h\in H\mid (id\otimes\pi)\Delta (h)=h\otimes 1\}$ is generated by $u, v$.
Since $c_{11}c_{22}-c_{12}c_{21}=1, $ an argument similar to the one used in Example \ref{exampleD8} shows that
 $\alpha_i=-1$
 for $1\leq i\leq 4.$

\begin{example}\rm
Let $H$ be a Hopf algebra generated as an algebra by $\{c_{ij}\}_{i=1, 2; j=1, 2}, u, v$ satisfying (\ref{relationQ8}) and the following relations:
$$
\C\odot^\prime\X=K(\X\odot\C),
\;\;u^2=v^2=0,\;\; uv+vu=0,$$
where
  $$
K=
\left(\begin{array}{cccc}
-1\\
&-1\\
&&-1\\
&&&-1
 \end{array}\right)
 .
 $$
 The coalgebra structure and antipode are given by (\ref{comultiQ8}), (\ref{SQ8}) and
 $$
 \Delta(u)=c_{11}\otimes u+c_{12}\otimes v+u\otimes 1,\;\;\varepsilon(u)=0,
 $$
 $$
\Delta(v)=c_{21}\otimes u+c_{22}\otimes v+v\otimes 1,\;\;\varepsilon(v)=0,
$$
$$
S(u)=-(e_{00}+\sqrt{-1}e_{21}-e_{20}-\sqrt{-1}e_{01})u-(\sqrt{-1}e_{30}+e_{31}-\sqrt{-1}e_{10}-e_{11})v,
$$
$$
S(v)=-(\sqrt{-1}e_{30}-e_{31}-\sqrt{-1}e_{10}+e_{11})u-(e_{00}-\sqrt{-1}e_{21}-e_{20}+\sqrt{-1}e_{01})v.
$$
One can show that $H\cong(\k\langle x,y\rangle/(x^2, y^2, (xy)^2+(yx)^2))^*\times (\k Q_8)^*$, and it is a link-indecomposable coradically graded Hopf algebra of tame corepresentaion type over $(\k Q_8)^*$.
\end{example}
\subsection{Hopf algebras of tame corepresentation type over $H_8$}
Note that the simple subcoalgebras in $H_8$ are $\k1, \k c, \k b, \k bc,  C$, where $C=\operatorname{span}\{x, bx, cx, bcx\}$. We give a corresponding basic multiplicative matrix $\C$ of $C$, where
\begin{eqnarray*}
\C=\left(\begin{array}{cc}
c_{11}&c_{12}\\
c_{21}&c_{22}
\end{array}\right)=
\frac{1}{2}\left(\begin{array}{cc}
x+bx&x-bx\\
cx-bcx&cx+bcx
 \end{array}\right).
\end{eqnarray*}
Suppose there exists a link-indecomposable coradically graded Hopf algebra $H$ of tame corepresentaion type over $H_8$ such that the invertible matrix $K$ in Lemma \ref{lem:invertibleK} is diagonal.
Namely, there exists an diagonal invertible matrix $K=(k_{ij})_{4\times 4}$ over $\k$ such that
$$
\C\odot^\prime\X=K(\X\odot\C),
$$
where $\X=
\left(\begin{array}{cc}
u\\
v
 \end{array}\right)$
 is a non-trivial $(\C, \k1)$-primitive matrix, and
$$
K=
\left(\begin{array}{cccc}
\alpha_1\\
&\alpha_2\\
&&\alpha_3\\
&&&\alpha_4
 \end{array}\right)
 .
$$
Consider $\Delta(c_{11}u)$ and $\Delta(uc_{11})$; by Lemma \ref{lem:invertibleK} $c_{21}u, c_{11}v$ are linearly independent, which implies $c_{11}c_{12}=c_{12}c_{11},$ a contradiction.
Thus there exists no link-indecomposable coradically graded Hopf algebra $H$ of tame corepresentaion type over $H_8$ such that the invertible matrix $K$ in Lemma \ref{lem:invertibleK} is diagonal.
However, such a Hopf algebra does exist when $K$ is not diagonal.
\begin{example}\emph{(}\cite[Definition 5.18]{Shi19}\emph{)}\rm
Let $H$ be a Hopf algebra generated as an algebra by $x, y, z, p_1, p_2$ with relations (\ref{relation1}) and
\begin{eqnarray*}
p_1^2=p_2^2=0,\;\;p_1p_2p_1p_2+p_2p_1p_2p_1=0,
\end{eqnarray*}
\begin{eqnarray*}
xp_1=p_1x,\;\; yp_1=p_1y,\;\;xp_2=-p_2x,\;\;yp_2=-p_2y,
\;\;
zp_1=-p_1z,\;\;zp_2=\sqrt{-1}p_2xz.
\end{eqnarray*}
The coalgebra structure and antipode of $H$ are given by (\ref{relation2}-\ref{relation4}) and
\begin{eqnarray*}
\Delta(p_1)=(f_{00}-\sqrt{-1}f_{11})z\otimes p_1+(f_{10}+\sqrt{-1}f_{01})z\otimes p_2+p_1\otimes 1,\;\;\varepsilon(p_1)=0,
\end{eqnarray*}
\begin{eqnarray*}
\Delta(p_2)=(f_{00}+\sqrt{-1}f_{11})z\otimes p_2+(f_{10}-\sqrt{-1}f_{01})z\otimes p_1+p_2\otimes 1,\;\;\varepsilon(p_2)=0,
\end{eqnarray*}
\begin{eqnarray*}
S(p_1)=-z(f_{00}-\sqrt{-1}f_{11})-z(f_{10}+\sqrt{-1}f_{01}) p_2,
\end{eqnarray*}
\begin{eqnarray*}
S(p_2)=-z(f_{00}+\sqrt{-1}f_{11}) p_2-z(f_{10}-\sqrt{-1}f_{01}) p_1,
\end{eqnarray*}
where $f_{ij}=\frac{1}{4}[1+(-1)^ix][1+(-1)^ky], i, j=0, 1$.
We know that $$\X=\left(\begin{array}{cc}
p_1+p_2\\
-\sqrt{-1}(p_1-p_2)
\end{array}\right)$$
is a non-trivial $(\C, 1)$-primitive matrix. In this case,
$$K=\left(\begin{array}{cccc}
-\frac{1}{2}&\frac{\sqrt{-1}}{2}&-\frac{\sqrt{-1}}{2}&\frac{1}{2}\\
-\frac{\sqrt{-1}}{2}&-\frac{1}{2}&-\frac{1}{2}&-\frac{\sqrt{-1}}{2}\\
\frac{\sqrt{-1}}{2}&-\frac{1}{2}&-\frac{1}{2}&\frac{\sqrt{-1}}{2}\\
\frac{1}{2}&\frac{\sqrt{-1}}{2}&-\frac{\sqrt{-1}}{2}&-\frac{1}{2}
\end{array}\right),$$
and we can show that $$H\cong (\k\langle x,y\rangle/(x^2, y^2, (xy)^2+(yx)^2))^*\times H_8.$$ This means that $H$ is of tame corepresentation type.
\end{example}

\section*{Acknowledgments}
The second author was supported by National Key R$\&$D Program of China 2024YFA1013802 and NSFC 12271243.

\section*{Declarations}
\subsection*{Author Contributions}
All authors contributed to all aspects of this project.
\subsection*{Competing interests}
The authors declare no competing interests.
\subsection*{Ethical Approval} Not applicable.
\subsection*{Availability of data and materials}
Not applicable.

\end{document}